\documentclass[leqno,11pt]{amsart}
\usepackage{amsmath,amssymb}
\usepackage{color}
\usepackage{comment}
\usepackage{todonotes}

\allowdisplaybreaks

\theoremstyle{plain}
\newtheorem{theorem}{Theorem}[section]
\newtheorem{lemma}[theorem]{Lemma}
\newtheorem{proposition}[theorem]{Proposition}

\newtheorem*{theorem-A}{Theorem A}
\newtheorem*{theorem-B}{Theorem B}
\newtheorem*{theorem-C}{Theorem C}
\newtheorem*{theorem-D}{Theorem D}
\newtheorem*{theorem-E}{Theorem E}
\newtheorem*{theorem-F}{Theorem F}
\theoremstyle{definition}

\newtheorem{remark}[theorem]{Remark}

\newtheorem{example}[theorem]{Example}

\newtheorem{lemmaalph}{Lemma}

\def\rn{\mathbb R\sp n}

\def\R{\mathbb R}
\def\N{\mathbb N}
\def\S{\mathbb S}

\def\M{\mathcal M}

\def\limsup{\operatornamewithlimits{lim\,sup}}

\newtoks\by
\newtoks\paper
\newtoks\book
\newtoks\jour
\newtoks\yr
\newtoks\pages
\newtoks\vol
\newtoks\publ
\newtoks\eds
\newtoks\proc
\newtoks\no
\def\ota{{\hbox{???}}}
\def\cLear{\by=\ota\paper=\ota\book=\ota\jour=\ota\yr=\ota
\pages=\ota\vol=\ota\publ=\ota}
\def\endpaper{\the\by, \textit{\the\paper},
{\the\jour} \textbf{\the\vol} (\the\yr), \the\pages.\cLear}
\def\endbook{\the\by, \textit{\the\book}, \the\publ.\cLear}
\def\endprep{\the\by, \textit{\the\paper}, \the\jour.\cLear}
\def\endproc{\the\by, \textit{\the\paper}, \the\publ, \the\pages.\cLear}
\def\name#1#2{#1 #2}
\def\et{ and }
\numberwithin{equation}{section}

\newcommand{\norm}[1]{{\left\vert\kern-0.25ex\left\vert\kern-0.25ex\left\vert #1
    \right\vert\kern-0.25ex\right\vert\kern-0.25ex\right\vert}}

\begin{document}

%\date{\today}

\title[Boundedness of fractional Orlicz-Sobolev functions]{Boundedness of functions in fractional Orlicz-Sobolev spaces }

\author {Angela Alberico, Andrea Cianchi, Lubo\v s Pick and Lenka Slav\'ikov\'a}

\address{Angela Alberico, Istituto per le Applicazioni del Calcolo ``M. Picone''\\
Consiglio Nazionale delle Ricerche \\
Via Pietro Castellino 111\\
80131 Napoli\\
Italy} \email{a.alberico@na.iac.cnr.it}

\address{Andrea Cianchi, Dipartimento di Matematica e Informatica \lq\lq U. Dini"\\
Universit\`a di Firenze\\
Viale Morgagni 67/a\\
50134 Firenze\\
Italy} \email{andrea.cianchi@unifi.it}

\address{Lubo\v s Pick, Department of Mathematical Analysis\\
Faculty of Mathematics and Physics\\
Charles University\\
Sokolovsk\'a~83\\
186~75 Praha~8\\
Czech Republic} \email{pick@karlin.mff.cuni.cz}

\address{Lenka Slav\'ikov\'a,
Department of Mathematical Analysis, Faculty of Mathematics and
Physics,  Charles University, Sokolovsk\'a~83,
186~75 Praha~8, Czech Republic}
\email{slavikova@karlin.mff.cuni.cz}
\urladdr{}

\subjclass[2000]{46E35, 46E30}
\keywords{Fractional Orlicz--Sobolev spaces;  boundedness of functions; smooth approximation;   Orlicz spaces; Orlicz-Lorentz spaces;  rearrangement-invariant spaces}

\begin{abstract} A necessary and sufficient  condition
% growth condition is determined on the  Young function underlying the definition of
for  fractional Orlicz-Sobolev spaces  to be continuously embedded into $L^\infty(\rn)$   is exhibited. Under the same assumption, any function from the relevant fractional-order spaces is shown to be continuous.  Improvements of this result are also offered.  They provide  the optimal Orlicz target space, and the optimal rearrangement-invariant target space in the embedding in question. These results complement those already available in the subcritical case,  where the embedding into $L^\infty(\rn)$ fails. They also augment   a classical embedding theorem for  standard fractional Sobolev spaces.
\end{abstract}

%\date{\today}

\maketitle

\section{Introduction }\label{sec1}

Fractional-order Orlicz-Sobolev spaces are associated with a positive non-integer smoothness parameter $s$, and with a Young function $A$ which  dictates a degree of  integrability. They generalize the   Gagliardo-Slobodeckii fractional Sobolev spaces, independently introduced in \cite{Gagliardo} and \cite{Slobodeckij}, and are defined in terms of a Luxemburg type seminorm
$|\,\cdot \,|_{s,A,\rn}$.

When $s\in (0,1)$, the seminorm $|\,\cdot \,|_{s,A,\rn}$ is built upon the functional defined as
\begin{equation}\label{intro1}
J_{s,A}(u) = \int_{\rn} \int_{\rn}A\left(\frac{|u(x)-u(y)|}{|x-y|^s}\right)\frac{\,dx\,dy}{|x-y|^n}
\end{equation}
for a measurable function $u$ in $\rn$.
The space $V_d^{s,A}(\rn)$ of those functions $u$ decaying near infinity in a weakest possible sense,
and
 such that $|u|_{s,A,\rn}<\infty$  will be considered. If $s\in (1,\infty)\setminus \N$, then the space $V_d^{s,A}(\rn)$  consists of all functions $u$, whose weak derivatives up to the order $[s]$ decay near infinity,   for which
$\big|\nabla ^{[s]}u\big|_{\{s\},A,\rn}<\infty$. Here, $\nabla ^{[s]}u$ denotes the vector of all weak derivatives of $u$  of order $[s]$,  the integer part of $s$, and $\{s\}= s-[s]$, the fractional part of $s$. In particular, $\nabla ^{[s]}u=\nabla ^0u=u$ if $s\in (0,1)$.
Precise definitions can be found in Section \ref{back}, where the necessary background is collected.  Customary fractional Sobolev spaces, which will be denoted by $V_d^{s,p}(\rn)$, are recovered with the choice $A(t)=t^p$, with $p\geq 1$.

As a part of a wealth of
 investigations on linear and nonlinear nonlocal equations of elliptic and parabolic type,   the analysis of nonlocal  problems driven by  possibly non-polynomial type nonlinearities has recently started attracting the attention of researchers   \cite{BBX, BaS, KKPF, SCdAB}.
 %Fractional-order Orlicz-Sobolev spaces provide a natural  functional framework for their solutions.

A sound theory of fractional Orlicz-Sobolev spaces --  a natural  functional framework for   these problems --
is of course crucial in connection with their study.  Properties of fractional Orlicz-Sobolev spaces are the subject of the contributions \cite{ACPS_lim0, ACPS_limit1, ACPS_frac, ACPS_campanato, Bon:19, BreitCia,  BonderSalort, KKPF}. They provide extensions of some aspects of the theory of the classical fractional Sobolev spaces, which has been developed over the years --  see e.g.  \cite{BFV, BBM, BBM_2002,BouPonVanSch, BraCin, BraSal, BrezisMiro1, BrezisMiro2, BrezisNguyen, CW, CFW, CostadeFigYang, DPV, DydaFrank, DydaVaha, FiFuMaMiMo, FMT, FJX, FrankSeiringer, Heuer, Liang, Ludwig, Mallick, MaranoMosco, MaSh2, MaSh1, MuNa1, PaPi, PaRu, SeegerTrebels, Tz, Zhou}.

In particular, as in the case of integer-order spaces,  embedding theorems  are a central issue in the fractional-order case.
Optimal   embeddings for the fractional Orlicz-Sobolev  space $V_d^{s,A}(\rn)$,  of the form
\begin{equation}\label{general}
V_d^{s,A}(\rn) \to Y(\rn),
\end{equation}
where $Y(\rn)$ is an Orlicz space, or, more generally, a   rearrangement-invariant space, were established in \cite{ACPS_frac}. Here, and in what follows, the arrow $\lq \lq \rightarrow "$ stands for continuous embedding.  The results of \cite{ACPS_frac} deal with every
\begin{equation}\label{s}
s\in (0,n)\setminus \N,
\end{equation}
 in the \lq\lq subcritical" growth regime for $A$ near infinity dictated by the condition
\begin{equation}\label{subcrit}
 \int ^\infty \left (\frac {t}{A(t)}\right)^{\frac s{n-s}} \; dt =\infty.
\end{equation}
Note that,   when $A(t)=t^p$ near infinity, condition \eqref{subcrit} amounts to assuming that $1\leq p \leq \frac ns$.

The present contribution is focused on the validity of embeddings \eqref{general} in the complementary \lq\lq supercritical" growth dominion, corresponding to  orders of smoothness $s$ satisfying  \eqref{s}  and Young functions growing so fast near infinity that
\begin{equation}\label{supercrit}
 \int ^\infty \left (\frac {t}{A(t)}\right)^{\frac s{n-s}} \; dt <\infty.
\end{equation}

We emphasize that  restriction \eqref{s}
 is indispensable when embeddings of the form \eqref{general} are in question. In fact,
our discussion  begins by  showing  that, whatever the rearrangement-invariant target space $Y(\rn)$ is, embedding  \eqref{general} can only hold provided that  $s$ is as in \eqref{s}
and $A$ decays so slowly  near zero that
\begin{equation}\label{indisp}
 \int_0\left (\frac {t}{A(t)}\right)^{\frac s{n-s}} \; dt <\infty.
\end{equation}
These assumptions will thus be kept in force throughout.

The core of this paper is a result
 asserting that
condition \eqref{supercrit} is  necessary and sufficient for the space $V_d^{s,A}(\rn)$ to be   embedded into $L^\infty(\rn)$, namely for \eqref{general} to hold with $Y(\rn)=L^\infty(\rn)$. The same condition  also turns out to be equivalent to the embedding of $V_d^{s,A}(\rn)$ into the space of continuous functions. In particular, these conclusions provide us with an embedding into $L^\infty(\rn)$ in case of Young functions which behave like $t^p$ near infinity, with $p>\frac ns$. By contrast, no embedding as in \eqref{general} can hold for these values of $p$ if $A(t)=t^p$ globally, namely for the standard homogeneous space $V_d^{s,p}(\rn)$, because of the failure of condition \eqref{indisp}.
Such an embedding is classically restored, provided that the homogeneous space is replaced by its non-homogeneous version, which consists of those functions in $V_d^{s,p}(\rn)$ all of whose derivatives up to the order $[s]$ belong to $L^p(\rn)$.

 Interestingly, condition \eqref{supercrit}  exactly matches an analogous necessary and sufficient condition for embeddings into $L^\infty(\rn)$ of integer-order Orlicz-Sobolev spaces (see \cite{Cal,cianchi_ASNS, Mabook, Talenti} for first-order spaces, and \cite{cianchi_forum} for the higher-order case), which is  reproduced by just setting $s$ equal to the order of the latter spaces in \eqref{supercrit}.

Although fundamental, these   conclusions merely provide information on \lq\lq local" properties of functions in $V_d^{s,A}(\rn)$. Since these functions are defined on the entire Euclidean space $\rn$ -- a domain with infinite Lebesgue measure --  their integrability properties \lq\lq near infinity" are also relevant. With this regard, results parallel to those obtained under assumption \eqref{subcrit} in \cite{ACPS_frac}  are also offered. They provide us with the optimal Orlicz target space,  and the optimal rearrangement-invariant target space  for embedding \eqref{general} to hold when condition \eqref{supercrit} is current.
%Importantly, these embeddings
%enhance the standard embedding $V_d^{s,p}(\rn)\to L^\infty(\rn)$, for $p>\frac ns$ recalled above, in that the target space is replaced by a strictly smaller rearrangement-invariant space.

The results outlined  so far are weaved with a general reduction principle  for embedding \eqref{general}, of independent interest, which is also established. This principle applies irrespective of whether \eqref{subcrit} or \eqref{supercrit} holds, and informs us about  the equivalence of any embedding of this kind to a one-dimensional Hardy type inequality.

The approach to embeddings \eqref{general} exploited in \cite{ACPS_frac} under assumption \eqref{subcrit} relies upon an extension argument to a half-space in $\mathbb R^{n+1}$.  It enables one to derive a Hardy type inequality, from which  subcritical embeddings for $V_d^{s,A}(\rn)$ into optimal Orlicz and rearrangement-invariant spaces follow. This method
does not seem to be adaptable   to deduce optimal conclusions  when the opposite condition  \eqref{supercrit} is in force.

We have instead to resort to a strategy which consists in deriving optimal supercritical embeddings from subcritical ones. This  technique  can be developed thanks to embeddings available for Orlicz-Sobolev spaces built upon arbitrary subcritical  Young functions, those of power type not been sufficient. In particular, unlike the usual argument exploited for classical non-homogeneous  fractional Sobolev spaces, such an approach
 avoids embeddings into  Campanato type spaces as an intermediate step. This is a critical point, since, although  sharp embeddings of this type can be obtained \cite{ACPS_campanato}, their use does not yield the optimal criterion for embeddings into $L^\infty(\rn)$ mentioned above.

\section{Background}\label{back}

Here, we recall  basic definitions and classical properties concerning  the function spaces involved in our discussion, as well as   fractional Orlicz-Sobolev embeddings in the subcritical  setting.

\subsection{Orlicz spaces and rearrangement-invariant   spaces}\label{rispaces}

 A  function  $A: [0,
\infty ) \to [0, \infty ]$ is called a \emph{Young function} if it is convex, non-constant, left-continuous and $A(0)=0$. Any function enjoying these properties admits the representation
\begin{equation}\label{young}
A(t) = \int _0^t a(\tau ) d\tau \quad \text{for $t \geq 0$}
\end{equation}
for some non-decreasing, left-continuous function $a: [0, \infty )
\to [0, \infty ]$ which is neither identically equal to $0$ nor to
$\infty$.
%Note that
%\begin{equation}\label{aA}
%t/2 \, a(t/2) \leq A(t) \leq t\, a(t) \quad \text{for $t\geq 0\,.$}
%\end{equation}
One has that
\begin{equation}\label{kt}
kA(t) \leq A(kt) \quad \text{for $k \geq 1$ and $t \geq 0$.}
\end{equation}
%A function $A$ is said to satisfy the $\Delta_2$--condition  if there exists a constant $c$ such that
%\begin{equation}\label{delta2}
%A(2t) \leq c A(t) \quad \text{for $t \geq 0$.}
%\end{equation}
The \emph{Young conjugate} $\widetilde{A}$ of $A$  is the Young function obeying
\begin{equation}\label{Atilde}
\widetilde{A}(t) = \sup \{\tau t-A(\tau ):\,\tau \geq 0\}  \quad \text{for $t\geq 0$.}
\end{equation}
On denoting by $a^{-1}$  the left-continuous (generalized) inverse of the
function $a$ appearing in \eqref{young}, the following formula holds:
\begin{equation}\label{youngconj}
\widetilde A(t) = \int _0^t a^{-1}(\tau ) d\tau \quad \text{for $t \geq 0$.}
\end{equation}
%One can show that
%\begin{equation}\label{AAtilde}
%t \leq A^{-1}(t) \widetilde A^{-1}(t) \leq 2\, t \quad \text{for $t\geq 0$,}
%\end{equation}
%where $A^{-1}$ and $\widetilde A^{-1}$ stand for the generalized
%right-continuous inverses of $A$ and $\widetilde A$, respectively.
%\par\noindent
%A Young function $A$ is said to belong to the class $\Delta _2$
%globally if there exists a positive constant $C$ such that
%\begin{equation}\label{delta2}
%A(2t) \leq C A(t)
%\end{equation}
%for $t \geq 0$. If there exists $t _0 >0$ such that inequality
%\eqref{delta2} just holds for  $t \geq t_0$, then we say that $A \in
%\Delta _2$ near infinity. The function $A$ is said to belong to the
%class $\nabla _2$ globally if there exists a  constant $K>1$ such
%that
%\begin{equation}\label{nabla2}
%A(Kt) \geq 2K A(t)
%\end{equation}
%for $t \geq 0$. Membership of $A$ to $\nabla _2$ near infinity is
%defined accordingly. One has that
%\begin{equation}\label{deltanabla}
%A \in \Delta _2 \quad \hbox{[near infinity]} \quad \hbox{if and only
%if} \quad \widetilde A \in \nabla _2  \quad \hbox{[near infinity]}.
%\end{equation}
 A  Young function $A$ is said to \emph{dominate} another Young function $B$ \emph{globally} [resp. \emph{near zero}] [resp. \emph{near infinity}]
   if there exist positive
 constants $c$ and $t_0$ such that
\begin{equation}\label{B.5bis}
B(t)\leq A(c t) \quad \text{for $ t\geq 0$ \, [for $0\leq  t\leq t_0$] \, [for $t\geq t_0$]}.
\end{equation}
%The function $A$ is said to dominate $B$ near infinity[resp. near zero] if there
%exists $t_0> 0$ such that \eqref{B.5bis} holds for $t \geq t_0$ [$t \leq t_0$].
The functions $A$ and $B$ are called \emph{equivalent  globally}, or \emph{near zero}, or \emph{near infinity},   if they dominate each other in the respective range of values of their argument.
\\ We shall write
$B\lesssim A$
to denote that $A$ dominates $B$, and
$A \simeq B$ to denote that $A$ is equivalent to $B$.
\\ By contrast, the relation $\lq\lq \approx "$ between two expressions will be used to denote that they are bounded by each other, up to positive multiplicative constants depending on appropriate  quantities.
\\
%The function $B$ is said to grow essentially more slowly near infinity than $A$ if
%\begin{equation}\label{nov 110}
%\lim _{t \to \infty} \frac{B(\lambda t)}{A(t)} =0
%\end{equation}
%for every $\lambda >0$. Note that condition \eqref{nov 110} is equivalent to
%\begin{equation}\label{nov 110bis}
%\lim _{t \to \infty} \frac{A^{-1}(t)}{B^{-1}(t)} =0.
%\end{equation}
%\\
The decay near $0$ of a Young function $A$ can be compared with that of a power function via its \emph{Matuszewska-Orlicz index at zero}, defined as
%
%
%Given a Young function $A$, we define the function $h_A^\infty\colon (0, \infty) \to [0, \infty)$ as
%\begin{equation*}
%	h_A^\infty (t) = \sup_{s>0} \frac{A^{-1}(st)}{A^{-1}(s)}
%    \quad \text{for $t>0$.}
%\end{equation*}
%The global lower and upper Boyd indices of $A$ are then defined as
%\begin{equation} \label{BIdef}
%    i_A^\infty = \sup_{1<t<\infty} \frac{\log t}{\log h_A^\infty(t)}
%        \quad\text{and}\quad
%    I_A^\infty	= \inf_{0<t<1} \frac{\log t}{\log h_A^\infty(t)}\,,
%\end{equation}
%respectively. One has that
%\begin{equation} \label{BIprop1}
%    1\le i_A^\infty \le I_A^\infty \le \infty.
%\end{equation}
%It can also be shown that
%\begin{equation} \label{BIprop2}
%    i_A^\infty	= \lim_{t\to\infty} \frac{\log t}{\log h_A^\infty(t)}
%        \quad\text{and}\quad
%    I_A^\infty	= \lim_{t\to 0^+} \frac{\log t}{\log h_A^\infty(t)}.
%\end{equation}
\begin{equation}\label{index}
   I_0(A) 	= \lim_{\lambda \to 0^+} \frac{\log {\color{black}\lambda}}{\log \Big(\limsup_{t\to0^+} \frac{A^{-1}(\lambda t)}{A^{-1}(t)} \Big)}.
\end{equation}
If $A$ vanishes only at $0$, then the following alternative expression for $I_0(A)$ holds:
\begin{equation}\label{indexbis}
I_0(A) = \lim_{\lambda \to \infty} \frac{\log \Big(\limsup _{t\to 0^+}\frac{A(\lambda t)}{A(t)}\Big)}{\log \lambda}.
\end{equation}
%
%.
%Recall that the upper Matuszewska-Orlicz  index $I(A)$ of a finite-valued Young function $A$ is defined as
%\begin{equation}\label{index}
%I(A) = \lim_{\lambda \to \infty} \frac{\log \Big(\sup _{t>0}\frac{A(\lambda t)}{A(t)}\Big)}{\log \lambda}.
%\end{equation}
%The Matuszewska-Orlicz index  $I_\infty (A)$  of $A$ near infinity  is defined analogously, with
%$\sup _{t>0}\frac{A(\lambda t)}{A(t)}$ replaced by $\limsup _{t \to \infty}\frac{A(\lambda t)}{A(t)}$.
%\par
%Let $\Omega$ be  a measurable subset of $\rn$, with $n\geq 1$.
\par Let us set
\begin{equation}\label{M}
\mathcal{M}(\rn)=\{u:\rn \to \R : \text{$u$ is  measurable}\},
\end{equation}
%and
%\begin{equation}\label{M+}
%\mathcal{M}_+(\rn)=\{ u\in \mathcal{M}(\rn) : u \geq 0\}\,.
%\end{equation}
%The notation
and define
$\mathcal{M}_d(\rn)$ as the subset of $\mathcal{M}(\rn)$ of those functions $u$ that decay near infinity, in the sense that all their level sets $\{|u|>t\}$ have finite Lebesgue measure for $t>0$. Namely,
\begin{equation}\label{Md}
\mathcal{M}_d(\rn)=\{ u\in \mathcal{M}(\rn) : |\{|u|>t\}|<\infty\,\, \text{for every $t>0$}\}\,,
\end{equation}
where $|E|$ stands for the Lebesgue measure of a set $E\subset \rn$.
%Of course, $\mathcal{M}_d(\Omega)= \mathcal{M}(\Omega)$ provided that $|\Omega|<\infty$.
\\
The \emph{Orlicz space} $L^A (\rn )$, built upon a Young function
$A$,   is the Banach
space of those  functions $u\in \mathcal M(\rn)$ making the
 \emph{Luxemburg norm}
\begin{equation}\label{lux}
 \|u\|_{L^A(\rn )}= \inf \left\{ \lambda >0 :  \int_{\rn }A
\left( \frac{|u|}{\lambda} \right) dx \leq 1 \right\}\,
\end{equation}
finite. In particular, $L^A (\rn )= L^p (\rn )$ if $A(t)=
t^p$ for some $p \in [1, \infty )$, and $L^A (\rn )= L^\infty
(\rn )$ if $A(t)=0$ for $t\in [0, 1]$ and $A(t) = \infty$ for
$t\in (1, \infty)$.
%When convenient for specific choices of $A$, we shall also adopt the notation $A(L)(\rn)$ to denote the Orlicz space $L^A(\rn)$.
\par\noindent
A version of  H\"older's inequality in Orlicz spaces tells us that
\begin{equation}\label{holder}
\int _{\rn} |u v|\,dx \leq 2\|u\|_{L^A(\rn )}
\|v\|_{L^{\widetilde A}(\rn )}
\end{equation}
 for every $u \in L^A(\rn )$ and $v\in L^{\widetilde
A}(\rn )$. Moreover,
\begin{equation}\label{holderconv}
\|v\|_{L^{\widetilde A}(\rn )} \leq \sup _{u\in L^A(\rn )} \frac{ \int _{\rn} |u v|\,dx}{\|u\|_{L^A(\rn )}} .
\end{equation}
If $A$ dominates $B$ globally, then
\begin{equation}\label{normineq}
\|u \|_{L^B(\rn )} \leq c \|u \|_{L^A(\rn )}
\end{equation}
for every $u \in L^A(\rn )$, where $c$ is the  constant appearing in inequality
\eqref{B.5bis} (with $t_0=0$).
%If $|\Omega|<\infty$   and $A$ dominates $B$ near infinity, then inequality
%\eqref{normineq} continues to hold for some constant $C$ depending also on $A$, $B$ and $|\Omega|$.
Thus, if $A$ is  equivalent to $B$ globally, then $L^A(\rn)= L^B(\rn)$, up to equivalent norms.
% The same is true even if $A$ and $B$ are just equivalent near infinity, provided that $|\o|<\infty$.
\par The Orlicz-Lorentz  spaces provide us with  a family of function spaces  which generalizes  the Orlicz spaces. Given  a Young function $A$ and a number $q\in \R$, we denote by
 $L(A,q)(\rn)$ the \emph{Orlicz-Lorentz space}  of those functions $u \in \mathcal M(\rn)$ for which the quantity
\begin{equation}\label{aug300}
	\|u\|_{L(A, q)(\rn)}
		= \big\|r^{-\frac{1}{q}}u^{*}(r)\big\|_{L^A(0,\infty)}
\end{equation}
is finite. Under suitable assumptions on $A$ and $q$, this quantity is a norm,  which renders  $L(A,q)(\rn)$  a (non-trivial) Banach space. This is the case, for instance, if  $q>1$ and
\begin{equation}\label{aug310}
\int^\infty \frac{A(t)}{t^{1+q}}\, dt < \infty\,,
\end{equation}
see \cite[Proposition 2.1]{cianchi-ibero}.

\par The Orlicz spaces and the Orlicz-Lorentz  spaces are special instances  of rearrangement-invariant spaces, whose definition rests upon that of decreasing
rearrangement. Recall that the
 \emph{decreasing
rearrangement} $u^{\ast}$ of a function $u\in \mathcal M (\rn)$ is given by
\begin{equation}\label{u*}
u^{\ast}(r) = \inf \{t\geq 0: |\{x\in \rn: |u(x)|>t \}|\leq r \} \quad \text{for $r\geq 0$.}
\end{equation}
In other words, $u^{\ast}$
is the (unique) non-increasing,
right-continuous function from $[0,\infty)$ into $[0,\infty]$
which is equidistributed with $u$.

%\iffalse
%
%Moreover, we define the function $u^{\ast \ast}: (0, \infty)  \to [0, \infty]$ as
%\begin{equation}\label{u**}
%u^{\ast \ast}(r) =
%\frac{1}{r}\int_{0}^{r}u^{\ast}(\varrho) d\varrho \quad  \text{for $r> 0$.}
%\end{equation}
%Notice that $u^* \leq u^{\ast \ast}$.
% The
%\emph{Hardy-Littlewood inequality} states that
%\begin{equation}\label{B.0}
%\int_{\rn}|uv| \,dx \leq \int_{0}^{\infty}u^{\ast}v^{\ast}\,dr
%\end{equation}
%for all functions $u, v \in \mathcal M(\rn)$. As a consequence, one also has that
%\begin{equation}\label{HLA}
%\int_{\rn}A(|uv|) \,dx \leq \int_{0}^{\infty}A(u^{\ast}v^{\ast})\,dr
%\end{equation}
%for every Young function $A$.
%
%\fi

 \par\noindent A \emph{rearrangement-invariant space} is a
Banach function space $X(\rn)$, in the sense of Luxemburg \cite[Chapter 1, Section 1]{BS},   such that
%\todo[inline]{A: I would go for rearrangement-invariant everywhere, as we did in our recent papers}
%\todo[inline]{Angela: OK Andrea! Done!}
\begin{equation}\label{B.1}
 \|u\|_{X(\rn)} = \|v \|_{X(\rn)} \quad \text{whenever $u^*=v^*$.}
 \end{equation}

\par\noindent
The \emph{representation space} $\overline{X}(0,\infty)$ of a rearrangement-invariant space $X(\rn)$
is defined as the unique rearrangement-invariant space on $(0,\infty)$ such that
\begin{equation}\label{B.3}
\|u \|_{X(\rn)} = \|u^{\ast} \|_{\overline{X}(0,\infty)}
\end{equation}
for every $u\in X(\rn)$.
\\ A basic property  tells us that, if $X(\rn)$ and $Y(\rn)$ are rearrangement-invariant spaces, then
\begin{equation}\label{inclusion-embedding}
X(\rn) \subset Y(\rn) \quad \text{if and only if} \quad X(\rn) \to Y(\rn).
\end{equation}

\subsection{Fractional Orlicz-Sobolev spaces}\label{fracorlicz}

%Assume  that $\Omega$ is an open subset of $\rn$.
Given $m \in \N$ and a  Young function $A$, we
denote by $V^{m,A}(\rn )$ the integer-order \emph{homogeneous Orlicz-Sobolev space} given by
\begin{equation}\label{homorliczsobolev}
V^{m,A}(\rn ) = \{ u\in \mathcal M(\rn):  \text{$u$ is   $m$-times weakly differentiable in $\rn$ and   $|\nabla ^m u| \in L^A(\rn)$}\}.
\end{equation}
The functional
$$\|\nabla ^m u\|_{L^A(\rn)}$$
defines a seminorm on the space $V^{m,A}(\rn )$.
%Here, $\nabla ^m u$ denotes the vector of all weak derivatives of $u$ of order $m$. If $m=1$, we also simply write $\nabla u$ instead of $\nabla^1 u$.
%Moreover, $W^{m,1}_{\rm loc} (\rn)$ is the standard Sobolev space of $m$-times weakly differentiable functions.
%The notation
%$W^{m,A}(\rn )$  is adopted for the classical Orlicz-Sobolev space defined by
%\begin{equation}\label{orliczsobolev}
%W^{m,A}(\rn ) = \{u \in V^{m,A}(\rn): |\nabla ^k u| \in L^A(\rn), \, k=0,1, \dots, m-1\}\,,
%\end{equation}
%where $\nabla ^0u$ has to be interpreted as $u$.
%The space $W^{m,A}(\rn)$ is a Banach space equipped with the
%norm
%$$\|u\|_{W^{m,A}(\rn )} = \sum _{k=0}^m  \|\nabla^k u
%\|_{L^A(\rn)}\,.$$
\\ As for fractional-order spaces, given $s\in (0,1)$,
%Define the modular functional $\Phi_{A,s}$ as
%\begin{equation}\label{10'}
%\Phi_{A,s}(u)=  \displaystyle \int_{\rn}\int_{\rn} A\left(\frac{| u(x) -  u(y)|}{|x-y|^{s}}\right)\, \frac{dx\,dy}{|x-y|^n}
%\end{equation}
%for $u \in\M (\rn)$.
the seminorm $|u|_{s,A, \rn}$ is  defined as
\begin{equation}\label{aug340}
|u|_{s,A, \rn}=  \inf\Big\{\lambda>0: J_{s,A}\Big(\frac u{\lambda}\Big)\leq 1\Big\}
%		= \inf\left\{\lambda>0: \int_{\rn} \int_{\rn}A\left(\frac{|u(x)-u(y)|}{\lambda|x-y|^s}\right)\frac{\,dx\,dy}{|x-y|^n}\le1\right\}\,.
\end{equation}
for $u \in \mathcal M (\rn)$, where $J_{s,A}$ is the functional given by \eqref{intro1}.
The \emph{homogeneous
 fractional Orlicz-Sobolev space} $V^{s,A}(\rn)$ is defined as
\begin{equation}\label{aug341}
	V^{s,A}(\rn) = \big\{u \in \mathcal M (\rn):  |u|_{s,A, \rn}<\infty\}\,.
\end{equation}
The definitions of the seminorm $|u|_{s,A, \rn}$ and of the space $V^{s,A}(\rn)$ carry over to vector-valued functions $u$ just by replacing the absolute value of $u(x)-u(y)$ by the Euclidean norm  of the same expression in the definition of the functional $J_{s,A}$.
\\
The subspace   of those functions in $V^{s,A}(\rn)$ which decay near infinity is denoted by $V^{s,A}_d(\rn)$. Thus,
\begin{equation}\label{nov100}
	V^{s,A}_d(\rn) =  \{ u\in V^{s,A}(\rn): |\{|u|>t\}|<\infty\,\, \text{for every $t>0$}\} \,.
\end{equation}
The definition of $V^{s,A}(\rn)$ is extended to all $s\in (0, \infty) \setminus \N$ in a standard way. On
denoting, as above, by $[s]$ and  $\{s\}$ the integer and the fractional part of $s$, respectively,
we define
\begin{equation}\label{aug343}
V^{s,A}(\rn ) = \{u \in \mathcal M(\rn): \text{$u$ is   $[s]$-times weakly differentiable in $\rn$ and  $ \nabla ^{[s]}u \in V^{\{s\}, A}(\rn)$}\}\,.
\end{equation}
In analogy with \eqref{nov100},  for  every $s\in (0, \infty) \setminus \N$ we set
\begin{equation}\label{nov100higher}
	V^{s,A}_d(\rn) =  \{ u \in V^{s,A}(\rn): |\{|\nabla ^k u|>t\}|<\infty\,\ \text{for  every $t>0$ and for $k=0,1, \dots ,[s]$}\} \,.
\end{equation}
If $s$ and  $A$  fulfill conditions \eqref{s} and \eqref{indisp}, then
the functional $\big|\nabla ^{[s]}u\big|_{\{s\},A, \rn}$ defines a norm on the space $V^{s,A}_d(\rn)$,
and the latter, equipped with this norm, is a Banach space. This is the content of Proposition \ref{banach}, Section \ref{proofmain}.

%The case of interest for our purposes is when
%\begin{equation}\label{indisp}
% s\in (0,n)\setminus \N \quad \text{and} \quad
% \int_0\left (\frac {t}{A(t)}\right)^{\frac s{n-s}} \; dt <\infty.
%\end{equation}
%These assumptions will be shown to be necessary for a  rearrangement-invariant space to exist into which the space $V_d^{s,A}(\rn)$  is continuously embedded -- see Theorem \ref{teo3} below.

%\todo[inline]{Angela: Why the marked sentence? The optimal embedding \eqref{may20} holds even in the supercritical case. }

When the
 subcritical growth condition \eqref{subcrit} is in force, optimal embeddings for  $V^{s,A}_d(\rn)$ take the following form.  The optimal Orlicz target  space for $V_d^{s,A}(\rn)$ is built upon the Young function
 $A_{\frac{n}{s}}$ defined, for $n \in \N$ and $s \in (0, n)\setminus \N$, as
%
%\begin{equation}\label{E:0'}
%	\int^{\infty}\left(\frac{t}{A(t)}\right)^{\frac{s}{n-s}}\,dt = \infty
%\end{equation}
%and
%\begin{equation}\label{E:0''}
%	\int_{0}\left(\frac{t}{A(t)}\right)^{\frac{s}{n-s}}\,dt < \infty.
%\end{equation}	
%Such an Orlicz target is defined in terms of the Young function
\begin{equation}\label{Ans}
A_{\frac{n}{s}} (t) = A(H^{-1}(t)) \quad \text{for $t\geq 0$,}
\end{equation}
where
\begin{equation}\label{H}
H(t) = \bigg(\int _0^t \bigg(\frac \tau{A(\tau)}\bigg)^{\frac
{s}{n-s}} d\tau\bigg)^{\frac {n-s}n} \quad \text{for $t \geq0$.}
\end{equation}
Indeed,  \cite[Theorems~6.1 and~7.1]{ACPS_frac} tell us that
\begin{equation}\label{may20}
V_d^{s,A}(\rn) \to L^{A_{\frac{n}{s}}}(\rn),
\end{equation}
and there exists a constant $c$ such that
\begin{equation}\label{may21}
\|u\|_{L^{A_{\frac ns}} (\rn)} \leq c  \big|\nabla ^{[s]}u\big|_{\{s\},A, \rn}
\end{equation}
for every $u \in V_d^{s,A}(\rn)$. Moreover,
the space $L^{A_{\frac{n}{s}}}(\rn)$ is the smallest possible target in \eqref{may20} in the class of all Orlicz spaces.
\\ The target space $L^{A_{\frac{n}{s}}}(\rn)$ in embedding \eqref{may20} can yet be enhanced if the realm of Orlicz spaces is abandoned, and also Orlicz-Lorentz spaces are allowed. Specifically, the
 Orlicz-Lorentz space  $L(\widehat A,\frac{n}{s})(\rn)$ comes into play, and is endowed with the norm defined, according to equation \eqref{aug300}, by
\begin{equation}\label{E:29}
	\|u\|_{L(\widehat A,\frac{n}{s})(\rn)}
		= \big\|r^{-\frac{s}{n}}u^{*}(r)\big\|_{L^{\widehat A}(0,\infty)}
\end{equation}
for $u \in \M (\rn)$, where
 $\widehat A$ is the Young function given by
\begin{equation}\label{E:1}
	\widehat A (t)=\int_0^t\widehat a (\tau)\,d\tau\quad\text{for $t\geq 0$},
\end{equation}
and
\begin{equation}\label{E:2}
	{\widehat a\,}^{-1}(t) = \left(\int_{a^{-1}(t)}^{\infty}
		 \left(\int_0^{\tau}\left(\frac{1}{a(\theta)}\right)^{\frac{s}{n-s}}\,d\theta\right)^{-\frac{n}{s}}\frac{d\tau}{a(\tau)^{\frac{n}{n-s}}}
				\right)^{\frac{s}{s-n}}
					\quad\text{for $t\ge0$}.
\end{equation}
One has that
\begin{equation}\label{may22}
V_d^{s,A}(\rn) \to L(\widehat A,\tfrac{n}{s})(\rn),
\end{equation}
and there exists a constant $c$ such that
\begin{equation}\label{may23}
\|u\|_{L(\widehat A,\frac{n}{s})(\rn)} \leq c  \big|\nabla ^{[s]}u\big|_{\{s\},A, \rn}
\end{equation}
for every $u \in V_d^{s,A}(\rn)$. Furthermore, the space $L(\widehat A,\frac{n}{s})(\rn)$ in embedding \eqref{may22} is  the smallest possible among all rearrangement-invariant spaces.
\\ Hence,
\begin{equation}\label{4.30}
 L(\widehat{A}, \tfrac ns )(\rn)  \to L^{A_{\frac ns}}(\rn)\,.
\end{equation}
By \cite[Proposition 4.1]{ACPS_frac}, the norm of embedding \eqref{4.30} depends only on $\frac ns$.

In particular, the proof of the optimality of embeddings \eqref{may20} and \eqref{may22} rests upon the fact that, given any rearrangement-invariant space $Y(\rn)$,
\begin{equation}\label{rednec}
\text{if \,$V_d^{s,A}(\rn) \to Y(\rn)$, \,then \, $\bigg\|\int_{r}^{\infty}f(\rho)\rho^{-1+\frac{s}{n}}\,d\rho \bigg\|_{\overline{Y}(0, \infty)}
		\leq c\|f\|_{L^A(0, \infty)}$   for every   $f\in L^A(0,\infty)$,}
\end{equation}
for some constant $c$.
\\ Property \eqref{rednec}  is proved in \cite[Lemmas 6.5 and 7.6]{ACPS_frac}. It
 is one of the two implications of the
 reduction principle contained in Theorem \ref{T:reduction_principle}, Section \ref{sec}, the novelty here being the reverse one.

\section{Main results }\label{sec}

As a preliminary for the embeddings of $V_d^{s,A}(\rn)$  to be offered, we state the necessity of conditions \eqref{s} and \eqref{indisp}  on $s$ and $A$.

%As mentioned in Section \ref{sec1}, conditions \eqref{s} and \eqref{indisp} are
%necessary   for any embedding  of $V_d^{s,A}(\rn)$ of the form \eqref{general} to hold.
%This is the content of the following preparatory result.

\begin{theorem}{\rm {\bf  [Admissible $s$ and $A$]}}
\label{teo3}
Let $s\in (0, \infty)\setminus \N$ and let $A$ be a Young function.  Assume that   embedding \eqref{general}
holds for some  rearrangement-invariant space $Y(\rn)$.  Then, $s$ and $A$ fulfill conditions \eqref{s} and \eqref{indisp}.
%
%Assume that either
%\begin{equation}\label{divintinf}
%s\in (0,n) \quad \text{and} \quad \int_0  \left (\frac {t}{A(t)}\right)^{\frac s{n-s}} \; dt =\infty,
%\end{equation}
%or
%\begin{equation}\label{s>n}
%s\in (n,\infty).
%\end{equation}
% Then,  no embedding of the form
%\begin{equation}\label{embY}
%V_d^{s,A}(\rn) \to  Y(\rn)
%\end{equation}
%can hold for any rearrangement-invariant space $ Y(\rn)$.
\end{theorem}

Having clarified the indispensability  of assumptions  \eqref{s} and \eqref{indisp}, we are ready to state our first main result. It tells us that the supercritical growth condition \eqref{supercrit} is necessary and sufficient
 for  the space $V_d^{s,A}(\rn)$ to be continuously embedded in $L^\infty(\rn)$.
 The relevant condition turns also to be equivalent to the embedding  of $V_d^{s,A}(\rn)$ into  the space $C^0(\rn)$ of continuous functions in $\rn$, equipped with the standard norm. The fact that a function in $V_d^{s,A}(\rn)$ belongs to
$C^0(\rn)$ has to be interpreted, as usual, in the sense that $u$ agrees a.e. in $\rn$ with a continuous function in  $\rn$.

\begin{theorem}{\rm {\bf [Embeddings into $L^\infty(\rn)$ and $C^0(\rn)$]}}
\label{teo_cont} Assume that $s$ and $A$ satisfy conditions \eqref{s} and \eqref{indisp}.
%Let $s\in (0, n)\setminus \N$ and let $A$ be a Young function such that
%\begin{equation}\label{2.8-0}
% \int_0\left (\frac {t}{A(t)}\right)^{\frac s{n-s}} \; dt <\infty.
%\end{equation}
Then, the following statements are equivalent:

\smallskip\par\noindent  (i)  Condition \eqref{supercrit} holds;
%\begin{equation}\label{2.8}
% \int ^\infty \left (\frac {t}{A(t)}\right)^{\frac s{n-s}} \; dt <\infty;
%\end{equation}

\smallskip\par\noindent
(ii) The embedding
\begin{equation}\label{linf}
V_d^{s,A}(\rn) \to  L^\infty(\rn)
\end{equation}
holds;

\smallskip\par\noindent (iii) The embedding
\begin{equation}\label{cont}
V_d^{s,A}(\rn) \to C^0(\rn)
\end{equation}
holds.
\\
Moreover, if condition \eqref{supercrit} is in force, then there exists a constant $c$ such that
\begin{equation}\label{201bis}
\|u\|_{L^\infty(\rn)} \leq c\, \left(\int_{\rn} \int_{\rn} A\left(\frac{| \nabla ^{[s]}u (x)-  \nabla ^{[s]}u(y)| }{|x-y|^{\{s\}}}\right)\, \frac{dxdy}{|x-y|^n}\right)^{\frac sn}
\end{equation}
for every $u \in V_d^{s,A}(\rn)$.
\end{theorem}

The norm in  $L^\infty(\rn)$, regarded as an Orlicz space,  is, loosely speaking, the strongest   \lq\lq locally" in $\rn$, but the weakest  \lq\lq near infinity". Indeed, membership of a function in $L^\infty(\rn)$ does not entail any information on the behaviour of the measure of its level sets when the levels approach zero.  Under the same assumption \eqref{supercrit},  the next result augments embedding \eqref{linf}, and provides us with the optimal Orlicz target space on $\rn$ into which the space $V_d^{s,A}(\rn) $ is continuously embedded.
\\ The Orlicz space in question is built upon the Young function
 $A_{\frac{n}{s}}$ defined as in \eqref{Ans}--\eqref{H} for the subcritical embedding \eqref{may20}.  The novelty is that, since we are now assuming that condition \eqref{supercrit} holds,
the function $H^{-1}$ has to be interpreted as  the generalized left-continuous inverse of $H$. In particular,
\begin{equation}\label{tinf}
H^{-1}(t)=\infty \quad  \text {for}\quad t > \int_{0}^\infty\left(\frac{\tau}{A(\tau)}\right)^{\frac{s}{n-s}}\,d\tau,
\end{equation}
and
\begin{equation}\label{inf}
A_{\frac{n}{s}}(t)=\infty
\end{equation}
for $t$ as in \eqref{tinf}. In particular,
\begin{equation}\label{subset}
L^{A_{\frac{n}{s}} }(\rn) \subsetneq L^\infty(\rn).
\end{equation}

\begin{theorem}
{\rm {\bf [Optimal Orlicz target space]}}
\label{teo2}
Assume that $s$ and $A$ satisfy conditions  \eqref{s}, \eqref{supercrit} and \eqref{indisp}.
%
%Let $s\in (0, n)\setminus \N$ and let $A$ be a Young function fulfilling assumptions \eqref{2.8-0}and \eqref{2.8}.
Let $A_{\frac ns}$ be the Young function defined by \eqref{Ans}--\eqref{H}.
Then,
\begin{equation}\label{2.10}
V_d^{s,A}(\rn) \to L^{A_{\frac ns}} (\rn),
\end{equation}
and there exists a constant $c$ such that
\begin{equation}\label{2.10norm}
\|u\|_{L^{A_{\frac ns}} (\rn)} \leq c  \big|\nabla ^{[s]}u\big|_{\{s\},A, \rn}
\end{equation}
for every $u \in V_d^{s,A}(\rn)$.
Moreover, the space $L^{A_{\frac ns}} (\rn)$ is optimal among all Orlicz target spaces in \eqref{2.10}.
%In particular, $L^{A_{\frac ns}} (\rn) \subsetneq L^\infty(\rn)$.
\end{theorem}

Let us point out that,
in view of property \eqref{inf}, only the asymptotic behaviour of the function  $A_{\frac ns}$ near zero has to be detected in identifying the space $L^{A_{\frac ns}} (\rn)$ in concrete applications.

  As recalled above, although   it is just stated  in the supercritical regime \eqref{supercrit} focused in this paper, Theorem \ref{teo2} also holds in the complementary situation when condition \eqref{subcrit} is in force.
\medskip
\par
%\note[inline]{Lubos: Should the comment concerning Theorem  \ref{teo1} be moved below the statement?}

The embedding provided by the next result stands to embedding \eqref{2.10} as embedding \eqref{may22} stands to embedding \eqref{may20}. Actually, it
tells us that embedding \eqref{2.10} can still be improved, provided that the class of admissible target spaces is further broadened as to include all rearrangement-invariant spaces. The optimal target in this class for embeddings of the space $V_d^{s,A}(\rn)$ can be obtained as the intersection of $L^\infty(\rn)$ with the Orlicz-Lorentz space appearing in \eqref{may22}.
%
%a space of Orlicz-Lorentz type. In a sense, the former accounts for the local behaviour of functions, whereas the latter describes the behaviour of functions near infinity.
%\\ The apropos Orlicz-Lorentz space is denoted by $L(\widehat A,\frac{n}{s})(\rn)$, and is endowed with the norm defined by
%\begin{equation}\label{E:29}
%	\|u\|_{L(\widehat A,\frac{n}{s})(\rn)}
%		= \|r^{-\frac{s}{n}}u^{*}(r)\|_{L^{\widehat A}(0,\infty)}
%\end{equation}
%for a measurable function $u$, where
% $\widehat A$ is the Young function given by
%\begin{equation}\label{E:1}
%	\widehat A (t)=\int_0^t\widehat a (\tau)\,d\tau\quad\text{for $t\geq 0$},
%\end{equation}
%and
%\begin{equation}\label{E:2}
%	{\widehat a\,}^{-1}(r) = \left(\int_{a^{-1}(r)}^{\infty}
%		 \left(\int_0^t\left(\frac{1}{a(\rho)}\right)^{\frac{s}{n-s}}\,d\varrho\right)^{-\frac{n}{s}}\frac{dt}{a(t)^{\frac{n}{n-s}}}
%				\right)^{\frac{s}{s-n}}
%					\quad\text{for $r\ge0$}.
%\end{equation}

\medskip

\begin{theorem}
{\rm {\bf [Optimal rearrangement-invariant target space]}}
\label{teo1}
Assume that $s$ and $A$ satisfy conditions  \eqref{s}, \eqref{supercrit} and  \eqref{indisp}. Let $\widehat A$ be the Young function associated with $A$ as in \eqref{E:1}--\eqref{E:2}.
Then,
\begin{equation}\label{2.9}
V_d^{s,A}(\rn) \to \big(L^\infty  \cap L\big ( \widehat{A} , \tfrac ns\big)\big)(\rn),
\end{equation}
and there exists a constant $c$ such that
\begin{equation}\label{2.9norm}
\|u\|_{ \left(L^\infty  \cap L( \widehat{A} , \frac ns)\right)(\rn)} \leq c  \big|\nabla ^{[s]}u\big|_{\{s\},A, \rn}
\end{equation}
for every $u \in V_d^{s,A}(\rn)$.
Moreover, the space $ \big(L^\infty  \cap L\big ( \widehat{A} , \tfrac ns\big)\big)(\rn)$ is optimal among all rearrangement-invariant target spaces  in \eqref{2.9}.
\end{theorem}

%A few comments on Theorem \ref{teo1} are in order.

\begin{remark}\label{rem-june}{\rm Recall that a customary  norm in  $\big(L^\infty \cap  L( \widehat{A} , \tfrac ns)\big)(\rn)$ is defined as
\begin{equation}\label{normsum}
\|u\|_{ L^\infty (\rn)} +\|u\|_{L( \widehat{A} , \frac ns)(\rn)}
\end{equation}
for $u \in \M (\rn)$.
%Since the space $ L( \widehat{A} , \tfrac ns)(\rn)$ enters embedding \eqref{2.9} only through its intersection with $L^\infty (\rn)$,
It turns out that, because of the presence of the norm in $L^\infty (\rn)$, only the decay of the function $ \widehat{A}$ at zero is relevant  in the definition of  $\big(L^\infty \cap  L( \widehat{A} , \tfrac ns)\big)(\rn)$. This is apparent when making use of an
%is shown in Lemma \ref{lemma3}, and follows from an
equivalent   norm in $\big(L^\infty \cap  L( \widehat{A} , \tfrac ns)\big)(\rn)$, given by
\begin{equation}\label{normequiv}
 \|u^* (r)\, \phi (r)\|_{L^{E_A} (0, \infty)}
\end{equation}
for $u \in \M (\rn)$, where
$E_A$ is a Young function such that
\begin{equation}\label{1.4}
E_A(t) \simeq
\begin{cases}
 \widehat{A}(t) & \hbox{near $0$}
 \\
 \infty & \hbox{near infinity}
 \end{cases}
\end{equation}
and the function
$ \phi : (0, \infty) \to [0, \infty)$ obeys
\begin{equation}\label{1.5}
\phi (r)= \min \Big\{1, r^{-\tfrac sn}\Big\}\qquad \hbox{for $r>0$}.
\end{equation}
The equivalence of the  norms \eqref{normsum} and \eqref{normequiv}
%\begin{equation*}
%    \|u\|_{ \left(L^\infty \cap L( \widehat{A} , \frac ns)\right)(\rn)}\approx \|u^* (r)\, \phi (r)\|_{L^{E_A} (0, \infty)},
%\end{equation*}
is established
in \cite[Proposition  2.1]{cianchi-ibero}.
\\ In this connection,
let us also mention that   $\widehat{A} \lesssim A$ in the sense of Young functions. Moreover,
\begin{equation}\label{equivhat}
 \widehat{A}(t)  \simeq  A(t) \quad \text{near zero}
\end{equation}
if and only if the upper Matuszewska-Orlicz index $I_0(A)$ of $A$ at zero, defined by \eqref{index},  satisfies
\begin{equation}\label{indexcond}
I_0(A) < \frac ns.
\end{equation}
Hence, under assumption \eqref{indexcond}, the space $L\big( \widehat{A} , \frac ns\big)(\rn)$  can be replaced  by $ L\big({A} , \frac ns\big)(\rn)$ in embedding \eqref{2.9}; namely, one has that
\begin{equation}\label{2.9bis}
V_d^{s,A}(\rn) \to \left(L^\infty  \cap L\big ({A} , \tfrac ns\big)\right) (\rn).
\end{equation}
}
\end{remark}

\begin{example}\label{EX:optimal-targets} Consider the space $V_d^{s,A}(\rn)$ associated with a Young function $A$
such that
\begin{equation}\label{E:young-domain}
    A(t)\simeq
        \begin{cases}
            t^{p_0} (\log \frac 1t)^{\alpha_0} & \quad \text{near zero}
                \\
            t^p  (\log t)^\alpha & \quad \text{near infinity.}
        \end{cases}
\end{equation}
In order for $A$ to be a Young function, the exponents appearing in equation \eqref{E:young-domain} are such that
either $p_0>1$ and $\alpha_0 \in \R$, or $p_0=1$ and $\alpha_0 \leq 0$, and either $p>1$ and $\alpha \in \R$ or $p=1$ and $\alpha \geq 0$.
\\
Let  $s\in(0,n)\setminus \N$. The function $A$ satisfies  the necessary   condition~\eqref{indisp}  from Theorem \ref{teo3} provided that
\begin{equation*}
    \text{either $1\le p_0<\frac{n}{s}$ and $\alpha_0\in\R$, or $p_0=\frac{n}{s}$ and $\alpha_0>\frac{n}{s}-1$.}
\end{equation*}
Moreover, the supercritical growth
assumption~\eqref{supercrit} amounts to requiring that
\begin{equation*}
    \text{either $p=\frac{n}{s}$ and $\alpha>\frac{n}{s}-1$, or $p>\frac{n}{s}$ and $\alpha\in\R$.}
\end{equation*}
Under these assumptions,
Theorem~\ref{teo2} tells us that  embedding~\eqref{2.10} and  inequality~\eqref{2.10norm} hold, with
\begin{equation}\label{E:orlicz-target-young-function}
  A_{\frac ns}(t)=\infty    \quad \text{near infinity,} \quad  \text{and} \quad   A_{\frac ns}(t)\simeq
    \begin{cases}
        t^{\frac {n{p_0}}{n-s{p_0}}} (\log \frac 1t)^{\frac {n\alpha_0}{n-s{p_0}}} & \quad \text{if $1\leq {p_0}< \frac ns$}
            \\
        e^{-t^{-\frac{n}{s(\alpha_0 +1)-n}}} & \quad  \text{if ${p_0}=\frac ns$ and $\alpha_0 > \frac ns -1$}
    \end{cases}
    \quad \text{near zero.}
\end{equation}
Moreover, the target space $L^{A_{\frac ns}}(\rn)$  is optimal among all Orlicz spaces.
\\ On the other hand, Theorem~\ref{teo1} implies that  embedding~\eqref{2.9} and   inequality~\eqref{2.9norm} hold,  where $\widehat A$ is any Young function such that
\begin{equation}\label{E:ri-target-young-function}
    \widehat A(t)\simeq
    \begin{cases}
        t^{p_0} (\log \frac 1t)^{\alpha_0} & \quad \text{if $1\leq {p_0}< \frac ns$}
            \\
        t^{\frac{n}{s}} (\log  \frac 1t)^{\alpha_0-\frac{n}{s}}
        & \quad  \text{if ${p_0}=\frac ns$ and $\alpha_0 > \frac ns -1$}
    \end{cases}
    \quad \text{near zero.}
\end{equation}
Moreover, the target space $\big(L^\infty  \cap L( \widehat{A} , \tfrac ns)\big)(\rn)$  is optimal among all rearrangement-invariant spaces.
\end{example}

%
%The necessity of conditions \eqref{indisp}
% in Theorems \ref{teo_cont} -- \ref{teo1} is the content of the next result. In fact,  it tells us that these conditions are in indispensable for the space $V_d^{s,A}(\rn)$ to be embedded into any rearrangement-invariant space.
%
%
%\begin{theorem}[\bf Necessary conditions for embeddings into rearrangement-invariant spaces]\label{teo3}
%Let $s\in (0, \infty)\setminus \N$ and let $A$ be a Young function.   Assume that either
%\begin{equation}\label{divintinf}
%s\in (0,n) \quad \text{and} \quad \int_0  \left (\frac {t}{A(t)}\right)^{\frac s{n-s}} \; dt =\infty,
%\end{equation}
%or
%\begin{equation}\label{s>n}
%s\in (n,\infty).
%\end{equation}
% Then,  no embedding of the form
%\begin{equation}\label{embY}
%V_d^{s,A}(\rn) \to  Y(\rn)
%\end{equation}
%can hold for any rearrangement-invariant space $ Y(\rn)$.
%\end{theorem}

%\medskip
%We emphasize that our assumptions in Theorems \ref{teo_cont} -- \ref{teo1} that $s\in (0,n)$ and that $A$ satisfies condition \eqref{2.8-0} are indispensable when discussing embeddings of the space $V_d^{s,A}(\rn)$ into rearrangement-invariant spaces. This is the content of the following result.
%

\bigskip
The proof of the results stated above is intertwined with a general characterization, of independent interest, of embeddings of the space $V^{s,A}_d(\rn)$ into rearrangement-invariant spaces $Y(\rn)$. It amounts to a reduction principle for such an embedding to a considerably simpler one-dimensional inequality for a Hardy type operator, depending only on $s$ and $n$, involving  the norms in the  spaces  $L^A(0, \infty)$ and $Y(0, \infty)$. This is the content of our last main result.

\begin{theorem}
{\rm {\bf [Reduction principle]}}
\label{T:reduction_principle}%{\rm{\bf [Higher-order reduction principle]}}
Let   $s \in (0,n)\setminus
\N$ and   let $A$ be a Young function.  Assume that  $Y(\rn)$ is a rearrangement-invariant space. Then, the following statements are equivalent:
\\ (i)
There exists a constant $c$ such that
\begin{equation}\label{E:i}
\|u\|_{Y(\rn)} \leq c  \big|\nabla ^{[s]}u\big|_{\{s\},A, \R^n}
\end{equation}
for every $u \in V^{s,A}_d(\rn)$.
\\ (ii)
There exists a constant $c$ such that
\begin{equation}\label{E:ii}
	 \bigg\|\int_{r}^{\infty}f(\rho)\rho^{-1+\frac{s}{n}}\,d\rho \bigg\|_{\overline{Y}(0, \infty)}
		\leq c\|f\|_{L^A(0, \infty)}
\end{equation}
 for every   $f\in L^A(0,\infty)$.
\end{theorem}

As already mentioned, the fact that embedding \eqref{E:i} implies inequality \eqref{E:ii}, which is also stated in \eqref{rednec}, was established in \cite[Lemmas~6.5 and 7.6]{ACPS_frac}. The novelty of Theorem \ref{T:reduction_principle}  is the reverse implication.

\section{Boundedness of fractional Orlicz-Sobolev functions}\label{boundedness}

This section is devoted to the proof of the following key result on the plain boundedness of functions from the space $V^{s,A}_d (\rn)$ in the supercritical regime.

\begin{theorem}
{\rm {\bf [Boundedness of fractional Orlicz-Sobolev functions]}}
\label{bound} Assume that $s$ and  $A$ satisfy  conditions \eqref{s},  \eqref{supercrit} and  \eqref{indisp}. Then,
\begin{equation}\label{inclusion}
V^{s,A}_d (\rn) \subset L^\infty(\rn).
\end{equation}
%
% If $u\in V^{s,A}_d (\rn)$, then $u \in L^\infty(\rn)$.
\end{theorem}

The  proof of Theorem \ref{bound} exploits the subcritical embedding \eqref{may20} to show that a function $u$ as in its statement belongs to any Orlicz space built upon a finite-valued Young function growing arbitrarily fast near infinity. This piece of information entails that, in fact, $u$ has to be essentially bounded. The   technical steps needed to implement this idea are split in some lemmas.

\begin{lemma}\label{lemma2}
Let $E$ be any finite-valued Young function. Then, there exists a continuously differentiable Young function $F$ such that
\begin{equation}\label{18}
F\geq E ,%\qquad \hbox{near infinity}\,,
\end{equation}
and
\begin{equation}\label{19}
\frac 1{F} \quad\hbox{is convex near infinity}\,.
\end{equation}
\end{lemma}

\begin{proof}
 Let  $G: [0, \infty) \to [0, \infty)$ be a function of the form
 \begin{equation}\label{20}
 G(t) = e^{\varphi (t)} -1 \qquad \hbox{for}\;\;t\geq 0\,,
 \end{equation}
 where $\varphi$ is  a twice continuously differentiable Young function such that $\varphi '(t)>0$ for $t>0$, and $\lim_{t\to \infty}\varphi ' (t) =\infty$. One can choose the function $\varphi$ in such a way that
\begin{equation}\label{may10}
G (t)\geq E(t) \quad \text{for $t \geq 0$\,.}
\end{equation}
For instance, the choice
$$\varphi(t) = \int_0^{2t}\frac 1\tau\int_0^{2\tau} \frac{E(\theta)+\theta^2} \theta\,d\theta \,d\tau \quad \text{for $t\geq0$}$$
is admissible, inasmuch as $\varphi(t) \geq E(t)+t^2>0$ for $t > 0$,   since the function $(E(t)+t^2)/t$ is non-decreasing. The latter property also ensures that the function
%\note[inline]{Lubos: Can we insert a short explanation of this inequality?}
 %Since $\varphi$ is a Young function, we have that
 $\varphi' $ is non-decreasing.\\
%\note[inline]{Lubos: It seems to me that this is obvious from the definition of $\varphi$ and the argument via Young function can be dropped.}
Define the function $\zeta : (0, \infty) \to (0, \infty)$ as
 $$\zeta (t) =\frac 1{\varphi '(t)} \qquad \hbox{for $t>0$}\,.$$
 The function $\zeta$ is positive, non-increasing and  continuously differentiable in $(0, \infty)$.  Also, $\displaystyle\lim_{t\to \infty} \zeta (t) =0$.\\
 Denote by $\widehat{\zeta }$ the convex envelope of $\zeta $. Such a function inherits properties from $\zeta $. Specifically, $\widehat{\zeta} > 0$  and  $\widehat{\zeta }$ is non-increasing  in $(0, \infty)$.  Moreover, $\displaystyle\lim_{t\to \infty}  \widehat{\zeta }(t) =0$. Inasmuch as $\zeta $ is continuously differentiable, $\widehat \zeta $ is continuously differentiable as well,
 \begin{equation}\label{23}
     \widehat{\zeta }^{\,\,'} \quad \hbox{is non-decreasing and}\quad \widehat{\zeta }^{\,\,'} <0,
 \end{equation}
 and
 \begin{equation}\label{23bis}
\lim_{t\to \infty}  \widehat{\zeta }^{\,\,'} (t) =0.
\end{equation}
 Let $\psi : [0, \infty) \to [0, \infty)$  be the function defined by
$$ \psi (t) =  \int_0^t\frac {d\tau}{\widehat{\zeta }(\tau)} \quad \text{for  $t\geq 0$}.
$$
Owing to the properties of the function  $\widehat{\zeta }$, one has that $ \psi $ is a  twice continuously differentiable Young function, and
 \begin{equation}\label{24}
 \psi '(t) = \frac 1{\widehat{\zeta }(t)}\qquad \hbox{for $t>0$}\,.
  \end{equation}
 Since $\widehat{\zeta } \leq \zeta $, we have that
 \begin{equation}\label{25}
 \psi(t) \geq \varphi (t) \qquad \hbox{for $t\geq 0$}\,.
  \end{equation}
  Moreover,
  \begin{equation}\label{26}
   \psi '' (t) = -\, \frac{ \widehat{\zeta }^{\,\,'}(t)}{\widehat{\zeta }(t)^2} = -\, \widehat{\zeta }^{\,\,'}(t) \, \psi ' (t)^2 \leq \psi ' (t)^2
 \qquad \hbox{near infinity,}
  \end{equation}
  where the last inequality holds by property \eqref{23bis}.
 % \todo[inline]{Angela: It seems to me that [26] and [22] assert  the same thing.
%Has [22] to be removed? }
%Hence,
%\begin{equation}\label{22}
%\psi '' \leq (\psi')^2\qquad \hbox{near infinity.}
%\end{equation}
\\
Now, define the Young function $F$ as
$$ F(t) =e^{\psi(t)} -1 \quad \hbox{for}\;\;t\geq 0.
$$
Owing to equations \eqref{20}, \eqref{may10} and  \eqref{25}, one has that
$$F(t) \geq E(t) \quad \text{for $t \geq 0$,}$$
whence property \eqref{18} holds.
%  \\ Inequality \eqref{22} thus holds with $\varphi$ replaced by $\psi$. Thus, given $E$, one can choose $\varphi$ such that
%  $$G \geq E  \qquad \hbox{near $\infty$}\,.$$
\\
We claim that property \eqref{19} is also fulfilled. Indeed, the latter is equivalent to
$$  \left(\frac 1{F}\right)'' \geq 0 \quad\hbox{near infinity,}$$
and this is in turn equivalent to
$$2\,(F')^2 - F\, F''\geq 0 \quad\hbox{near infinity,}$$
namely
\begin{equation}\label{may18}
\frac {F}{F '} \leq \frac{2F'}{F^{''}} \quad \hbox{near infinity.}
\end{equation}
 Since
 \begin{equation*}
    \frac{F}{F'}  = \frac 1{\psi '} - \frac 1{\psi ' \, e^{\psi}}\leq \frac 1{\psi'}
 \end{equation*}
 and
 \begin{equation*}
    \frac{F'}{F''}  = \frac {\psi'}{\psi '' + (\psi')^2},
 \end{equation*}
inequality \eqref{may18} follows from the fact that, by equation \eqref{26},
 $$\frac 1{\psi '} \leq \frac{2\, \psi'}{\psi '' + (\psi ')^2}\quad \hbox{near infinity.}$$
\end{proof}
%
%
%\par\noindent
%We have that
% \begin{align}\label{21}
% &\frac 1{\Psi} \quad \hbox{is convex near infinity} \quad \Longleftrightarrow \quad  \left(\frac 1{\Psi}\right)^{''} \geq 0 \quad\hbox{near $\infty$}
% \\&
% \quad \Longleftrightarrow \quad 2\,(\Psi')^2 - \Psi \, \Psi^{''} \geq 0 \quad\hbox{near $\infty$}
% \quad \Longleftrightarrow \quad
% \frac {\Psi}{\Psi '} \leq 2\, \frac{\Psi '}{\Psi ^{''}} \quad \hbox{near $\infty$}\,.\nonumber
% \end{align}
% Since
% \begin{equation*}
%    \frac{\Psi}{\Psi '}  = \frac 1{\varphi '} - \frac 1{\varphi ' \, e^{\varphi}}\leq \frac 1{\varphi'}
% \end{equation*}
% and
%
% \begin{equation*}
%    \frac{\Psi'}{\Psi ^{''}}  = \frac {\varphi '}{\varphi ^{''} + (\varphi ')^2}\,,
% \end{equation*}
% property \eqref{21} will follow if $\varphi$ fulfils
% $$\frac 1{\varphi '} \leq \frac{2\varphi '}{\varphi^{''} + (\varphi ')^2}\,,$$
% namely
%\begin{equation}\label{22}
%\varphi ^{''} \leq (\varphi')^2\qquad \hbox{near $\infty$}\,.
%\end{equation}
%
%
%
%
%  On replacing $\varphi$ by $\psi$, thanks to \eqref{25}, we have that the function $\overline{\Psi}$, defined by
%  $$\overline{\Psi}(t)= e^{\overline{\varphi}(t)} - 1 \,,$$
%  also satisfies
% $$ \overline{\Psi} \geq E \qquad \hbox{near $\infty$}\,,$$
% and, owing to \eqref{26}, also \eqref{21}.

\begin{lemma}\label{lemma1}
Assume that $s$ and  $A$ satisfy  conditions \eqref{s},  \eqref{supercrit} and  \eqref{indisp}.
Let $E$ be any finite-valued Young function. Then, there exists a finite-valued Young function $B$ such that
\begin{align}
&\hbox{$B \lesssim A$  globally and $A\simeq B$ near $0$\,;}\label{2}
\\ \nonumber
\\
&\hbox{$\displaystyle\int^\infty \left( \displaystyle\frac t{B(t)} \right)^{\frac s{n-s}} \; dt = \infty $\,;}\label{3}
\\ \nonumber
\\
&\hbox{$E \lesssim B_{\frac ns}$ near infinity.}
%
%$B_{\frac ns}$ dominates $E$ near $\infty$ (in the sense of Young function)\,.}
\label{4}
\end{align}
%\begin{enumerate}
%\item[(i)]{$A$ dominates $B$ globally and $A\approx B$ near $0$;}
%\medskip
%
%    \item[(ii)]{$\displaystyle\int^\infty \left( \displaystyle\frac t{B(t)} \right)^{\frac s{n-s}} \; dt = \infty $;}
%\medskip
%
%        \item[(iii)]{$B_{\frac ns}$ dominates $E$ near $\infty$ (in the sense of Young function).}
%\medskip
%\end{enumerate}
Here, $B_{\frac ns}$ denotes the Sobolev conjugate of $B$, defined as in \eqref{Ans}--\eqref{H}.
\end{lemma}

\begin{proof}  Let $a$ be the function from equation \eqref{young}, and let $b$ be the function appearing in a parallel equation with $A$ replaced by $B$.
By \cite[Lemma 2.3]{cianchi-ibero}, assumption \eqref{indisp} is equivalent to
\begin{equation}\label{1'}
\int_0 \frac{a^{-1}(t)}{t^{\frac{n}{n-s}}} \; dt < \infty ,
%\quad \left(\Longleftrightarrow \int^\infty \frac{\widetilde{A}(t)}{t^{1+
%\frac{n}{n-s}}
%}\; dt <\infty \right)
\end{equation}
and condition \eqref{3} is equivalent to
\begin{equation}\label{3'}
\int^\infty \frac{b^{-1}(t)}{t^{{\color{black}\frac{n}{n-s}}}} \; dt = \infty .
%\quad \left(\Longleftrightarrow \int^\infty \frac{\widetilde{B}(t)}{{\color{black}t^{1+{\frac{n}{n-s}}}}}\; dt =\infty \right)\,,
\end{equation}
We make use of an equivalent expression for the function $B_{{\color{black}\frac ns}}$, which tells us that
\begin{equation*}
B_{{\color{black}\frac ns}} (t)\simeq \left( t \, M ^{-1} \big(t^{\color{black}\frac n{n-s}}\big)\right)^{{\color{black}\frac n{n-s}}}\quad \text{for $t\geq 0$,}
\end{equation*}
where the function $ M : [0, \infty) \to [0, \infty)$ is defined by
\begin{equation*}
 M (t) = \int_0^t \frac{b^{-1}(\tau)}{\tau^{{\color{black}\frac n{n-s}}}}\;d\tau \quad \text{for $t\geq 0$,}
\end{equation*}
see \cite[Lemma 2]{cianchi_pacific}.
%\todo[inline]{Lenka: Should we assume that the integral above is finite when $b$ is replaced by $a$? This seems to be used and I don't see how it would a priori be satisfied, unless we assume it throughout without specifically mentioning it each time.}
%\todo[inline]{Lenka: Use a different variable than $s$?}
Thus, it suffices to produce a finite-valued Young function $B$ satisfying conditions \eqref{2} and \eqref{3}, and the following inequality, equivalent to \eqref{4}:
\begin{equation}\label{5}
  \left( t \, M ^{-1} \big(t^{{\color{black}\frac n{n-s}}}\big)\right)^{{\color{black}\frac n{n-s}}}\geq E(c\, t) \qquad \hbox{near infinity}
\end{equation}
for some positive constant  $c$.
\\ Now, observe that, if $F$ is a
 {\color{black} finite-valued} Young function  such that
\begin{equation}\label{7}
F(t) \geq E(t) \qquad \hbox{near infinity},
\end{equation}
then
\begin{equation*}
  t^{{\color{black}\frac n{n-s}}}\, F\big(t^{{\color{black}\frac n{n-s}}}\big)^{{\color{black}\frac ns}} \geq E(t)\qquad  \hbox{near infinity}\,.
\end{equation*}
Thereby,
 inequality \eqref{5} will follow if we show that
\begin{equation}\label{6}
      \left( t \, M  ^{-1} \big(t^{{\color{black}\frac n{n-s}}}\big)\right)^{{\color{black}\frac n{n-s}}} \geq t^{{\color{black}\frac n{n-s}}}\,  F\big(t^{{\color{black}\frac n{n-s}}}\big)^{{\color{black}\frac ns}} \qquad \hbox{near infinity}
\end{equation}
for some Young function $F$ satisfying inequality \eqref{7}.
\\ By Lemma \ref{lemma2}, there exists  a continuously differentiable Young function $F$ for which properties \eqref{7} and \eqref{19} hold.
%such that the function
%\begin{equation}\label{8}
%\frac 1{F(t)} \qquad \hbox{is convex   near infinity.}
%\end{equation}
Therefore, it suffices to exhibit a Young function $B$ satisfying properties \eqref{2}, \eqref{3} and \eqref{6}   for some Young function $F$ fulfilling \eqref{19}.
To this purpose,
recall that
\begin{equation}\label{9}
  \widetilde{A}(t)=\int_0^t a^{-1} (\tau)\;d\tau \quad  \quad \text{and}  \quad \quad \widetilde{B}(t)=\int_0^t b^{-1} (\tau)\;d\tau \quad \text{for $t\geq 0$.}
\end{equation}
We define the function $B$ via \eqref{9}, with
$$
 b^{-1}(t) = a^{-1}(t) + \eta (t) \qquad \text{for } t\geq 0\,,
$$
where the function $\eta : [0, \infty) \to [0, \infty) $ is such that
{\color{black}\begin{equation}\label{12}
\eta(t) \, \hbox{is non-decreasing,  and} \,\, \eta(t)=0 \,\, \text{near $0$.}
\end{equation}}
This choice ensures that condition \eqref{2} is fulfilled, for this condition is equivalent to
{\color{black}\begin{equation}\label{10}
  \text {$\widetilde{A} \lesssim \widetilde{B}\quad$ globally and  $\,\,\widetilde{A}  \simeq  \widetilde{B}\quad$ near $0$.}
\end{equation}}
In order that $B$  satisfies the remaining desired conditions, one can require that the function $\eta$ also  obeys
\begin{equation}\label{13}
\begin{cases}
\eta(t) =0 & \quad \text{if $0\leq t \leq t_0$}
\\ 	\\
 \displaystyle \int_{t_0}^t \frac{\eta(\tau)}{\tau^{{\color{black}\frac n{n-s}}}} \; d\tau = \frac 12 \, F^{-1}\big( t^{{\color{black}\frac s{n-s}}}\big) - \frac 12 \,F^{-1}\big( t_0^{{\color{black}\frac s{n-s}}}\big)&  \quad \text{if $t >t_0$}
\end{cases}
\end{equation}
for some $t_0>0$ to be chosen later.
\\
Indeed, equation  \eqref{13} implies  condition \eqref{3'}, and hence \eqref{3}, since
\begin{equation}\label{14}
\int_{t_0}^{\infty} \frac{b^{-1} (\tau)}{\tau^{{\color{black}\frac n{n-s}}}} \; d\tau =
\int_{t_0}^{\infty} \frac{a^{-1} (\tau)}{\tau^{{\color{black}\frac n{n-s}}}} \; d\tau +
\int_{t_0}^{\infty} \frac{\eta (\tau)}{\tau^{{\color{black}{\frac n{n-s}}}}} \; d\tau \geq \lim_{t\to \infty}\; \frac 12 \, F^{-1}\Big(t^{{\color{black}\frac s{ n-s}}}\Big) {\color{black}- \frac 12 \,F^{-1}\big( t_0^{\frac s{n-s}}\big)=\infty}\,.
\end{equation}
Moreover, equations \eqref{13}, \eqref{14} and \eqref{1'} entail that
\begin{equation}\label{15}
M  (t) =
\int_{0}^t \frac{a^{-1} (\tau)}{\tau^{{\color{black}{\frac n{n-s}}}}} \; d\tau +
\int_0^t \frac{\eta (\tau)}{\tau^{{\color{black}{\frac n{n-s}}}}} \; d\tau \leq
F^{-1}\Big(t^{{\color{black}\frac s{ n-s}}}\Big) \qquad \hbox{near infinity}\,,
\end{equation}
whence
\begin{equation}\label{16}
M^{-1}  (t) \geq F(t)^{{\color{black}\frac{n-s}{s}}} \qquad \hbox{near infinity}\,.
\end{equation}
Raising both sides of inequality \eqref{16} to the power {\color{black}$\frac n{n-s}$}, and multiplying through the resultant  inequality by $t^{{\color{black}\frac n{n-s}}}$ yield \eqref{6}.
\\
Thus, it  is only left  to show that a function $\eta$ fulfilling \eqref{12} and \eqref{13} does exist.  Equation \eqref{13} is equivalent to
\begin{equation}\label{17}
\eta(t)= \begin{cases} 0 & \quad \text{if $0\leq t \leq t_0$}
\\  \\
  \displaystyle\frac {s}{2(n-s)}\, \frac {t^{\frac {2s}{n-s}}}{F'\Big(F^{-1}\big ( t^{\frac s{n-s}}\big)\Big)} & \quad \text{if $t > t_0$.}
\end{cases}
\end{equation}
On the other hand, equation \eqref{12} will  be verified by showing that the function $\eta$ is non-decreasing for sufficiently large $t_0$. In view of equation \eqref{17}, this property will follow if we show that
\begin{equation}\label{may6}
\frac{F ^2}{F'} \quad \text{is non-decreasing near infinity.}
\end{equation}
Property \eqref{may6} is trivially equivalent to the fact that $\frac{F'}{F^2}$ is non-increasing  near infinity, namely to the fact that $ \big(-\frac1{F} \big)'$ \,\, is non-increasing  near infinity. The latter property is in turn equivalent to the concavity of the function $-\frac1{F}$ near infinity, and, hence, to the convexity of the function $\frac1{F}$ near infinity.  This is true, thanks to equation \eqref{19}.
%On the other hand,  the following assertions are equivalent:
%\begin{center} $\displaystyle \frac{F (s)^2}{F'(s)}$ \,\,  is non decreasing near infinity \end{center}
%%\begin{center} if and only if \end{center}
%\begin{center}$\displaystyle \frac{F' (s)}{F(s)^2}$  \,\, is non increasing  near infinity\end{center}
%%\begin{center} if and only if \end{center}
%\begin{center} $\displaystyle  \left( -\frac1{F(s)} \right)'$ \,\, is non increasing  near infinity\end{center}
%%\begin{center} if and only if \end{center}
%\begin{center} $  \displaystyle-\frac1{F(s)}$   \,\, is concave  near infinity\end{center}
%%\begin{center} if and only if \end{center}
%\begin{center} $ \displaystyle\frac1{F(s)}$  \,\, is convex  near infinity.\end{center}
%Hence, the conclusion follows, since the last assertion is true, thanks to equation \eqref{19}.
%%
%\begin{align*}
%   & \begin{center}\hbox{$\frac{F (s)^2}{F'(s)}$ \,\,  is non decreasing}\end{center}
%\\ & \hbox{if and only if}
%\\ &
%    \hbox{$\frac{F' (s)}{F(s)^2}$  \,\, is non increasing}
%		\\ &\;\; \text{if and only if}
%    \left( -\frac1{\Psi(s)} \right)'\,\, \text{is non increasing} \;\; \text{if and only if}
%     -\frac1{\Psi(s)} \,\, \text{is concave} \;\; \text{if and only if}
%     \frac1{\Psi(s)} \,\, \text{is convex}
%\end{align*}
%and this holds thanks to \eqref{8}.
\end{proof}

\begin{lemma}\label{lemma4}
Let $u \in \mathcal M(\rn)$ be  such that $u\in L^E (\rn)$ for every finite-valued Young function $E$ vanishing near $0$. Then, $u\in L^\infty (\rn)$.
\end{lemma}
\begin{proof}
Our assumption on the function $u$ ensures that
\begin{equation}\label{may1}
\int_{\rn} F(|u|)\, dx < \infty
\end{equation}
for every finite-valued Young function $F$ vanishing near $0$. To verify this assertion, observe that, given any function $F$ with this property, the function $E$ defined as
$$E(t) = F(t^2) \quad \text{for $t \geq 0$}$$
is also a Young function vanishing near $0$, and
\begin{equation}\label{may2}
\lim_{t \to \infty}\frac{E(\lambda t)}{F(t)} = \infty
\end{equation}
for every $\lambda >0$. Indeed, by property \eqref{kt},
$$
    \frac{E(\lambda t)}{F(t)} = \frac{F({\color{black}{\lambda^2}} t^2)}{F(t)}\geq \frac{{\color{black}{\lambda^2}} t F(t)}{F(t)}= {\color{black}{\lambda^2}} t\quad\quad \text{if $t\geq \frac 1{\color{black}{\lambda^2}}$.}
$$
Since $u \in  L^E (\rn)$,  there exists $\lambda >0$ such that
\begin{equation}\label{may3}
\int_{\rn} E(\lambda |u|)\, dx < \infty.
\end{equation}
Inasmuch as the function $E$ is finite-valued, there exists $t_0>0$ such that $|\{|u|>t\}|<\infty$ for  $t\geq t_0$, and property \eqref{may1} follows via equations \eqref{may2} and \eqref{may3}.
\\ Now, assume, by contradiction, that $u\notin L^\infty (\rn)$. Thus, the function $\mu : (0, \infty) \to [0,\infty)$, defined as
$$\mu(t) = |\{|u|>t\}| \quad \text{for $t>t_0$,}$$
is  non-increasing and such that
$$\mu (t) >0 \quad \text{for $t>t_0$.}$$
Hence, the function $F: [0, \infty) \to [0,\infty)$, given by
$$F(t) = \begin{cases} 0 &\quad \text{if $0 \leq t \leq t_0$}
\\
\\
\displaystyle\int_{t_0}^t\frac{d\tau}{\mu(\tau)}  &\quad \text{if $t>t_0$,}
\end{cases}
$$
is a finite-valued Young function vanishing in  $[0,t_0]$. Also,
\begin{equation}\label{may4}
\int_{\rn} F(|u|)\, dx = \int_0^\infty F'(t) \mu(t) \, dt = \int_{t_0}^\infty dt = \infty,
\end{equation}
a contradiction to equation \eqref{may1}.
\end{proof}

\begin{proof}[Proof of Theorem \ref{bound}]
Fix any finite-valued Young function $E$, and let $B$ be any Young function as in the statement of Lemma \ref{lemma1}. Define a Young function $E_0$ in such a way that
\begin{equation}\label{4.11}
E_0(t) =
\begin{cases}
 0 & \hbox{near $0$}
 \\
 E(t) & \hbox{near infinity}\,.
 \end{cases}
 \end{equation}
By equation \eqref{4},
\begin{equation*}
  B_{\frac ns} \quad \hbox{dominates} \quad E_0 \;\; \hbox{globally}  \,.
\end{equation*}
Hence,
\begin{equation}\label{4.12}
L^{B_{\frac ns}}(\rn)\to L^{E_0}(\rn) \,.
\end{equation}
Thanks to embedding \eqref{may20} with $A$ replaced by $B$,
\begin{equation}\label{4.13}
V_d^{s,B} (\rn)\to L^{B_{\frac ns}}(\rn) \,.
\end{equation}
On the other hand, property \eqref{2} implies that
\begin{equation}\label{4.14}
V_d^{s,A} (\rn)\to V_d^{s,B} (\rn) \,.
\end{equation}
Combining embeddings \eqref{4.12}-\eqref{4.14} yields:
\begin{equation}\label{4.15}
V_d^{s,A} (\rn)\to L^{E_0} (\rn) \,.
\end{equation}
 Owing to the arbitrariness of the Young function $E$, the latter embedding entails inclusion \eqref{inclusion}, via Lemma \ref{lemma4} .
%\begin{equation}\label{4.16}
%V_d^{s,A} (\rn) \subset L^\infty (\rn) \,.
%\end{equation}
\end{proof}

\section{Smooth approximation}\label{smooth}

An approximation argument by continuous functions is needed in the proof of embedding \eqref{cont} into the space of continuous functions.  Approximation for functions   $u \in V_d^{s,A}(\rn)$ in the  seminorm $\big|\nabla ^{[s]} u\big|_{\{s\},A,\rn}$ is only possible under the additional assumption that the Young function $A$ satisfies the $\Delta_2$-condition, an assumption which is not required in our results. By contrast,
the approximation theorem to be established in this section shows that modular approximation, namely approximation of $\nabla ^{[s]} u$ with respect to the functional $J_{\{s\},A}$ defined as in \eqref{intro1}, is possible for every finite-valued Young function $A$.

\begin{theorem}
{\rm {\bf [Modular smooth approximation]}}
 \label{prop2}  Let $s\in (0, {\color{black}\infty})\setminus \N$ and let $A$ be a finite-valued Young function. Let  $J_{\{s\},A}$ be the   functional defined as in  \eqref{intro1}.
%Denote
%\begin{equation}\label{10}
%\Phi_{A,s}(u)=  \displaystyle \int_{\rn}\int_{\rn} A\left(\frac{|\nabla^{[s]}u(x) - \nabla^{[s]} u(y)|}{|x-y|^{\{s\}}}\right)\, \frac{dx\,dy}{|x-y|^n}
%\end{equation}
%\medskip
%for those $u\in V_d^{s,A}(\rn)$, for which the expression on the right hand side makes sense.
%\todo[inline]{Lenka: Doesn't this make sense for any $u\in V_d^{s,A}(\rn)$? The integral may be infinite but I think it still makes sense.}
Assume that  $u\in V_d^{s,A}(\rn)$. Then,  there exist $\lambda >0$ and a sequence $\{u_j\} \subset C^\infty(\rn)$ such that
\begin{equation}\label{11}
\lim_{j\to \infty} J_{\{s\},A}\left (\frac{\nabla^{[s]} u_j -\nabla^{[s]} u}{\lambda}\right) =0\,.
\end{equation}
\end{theorem}

The approximating sequence announced in the statement of Theorem \ref{prop2} will be obtained via
convolutions defined as follows.
Let $\varrho \in C_0^\infty (\rn)$ be a nonnegative function   such that
\begin{equation}\label{A2}
\int_{\rn} \varrho (x) \; dx =1\qquad \hbox{and}\qquad {\rm supp}\, \varrho\subset \mathcal B,
\end{equation}
where $\mathcal B$ denotes the unit ball centered at $0$.
For each $\varepsilon >0$, set
\begin{equation}\label{A3}
\varrho _{\varepsilon}(x) = \varepsilon^{-n} \varrho \left(\frac x\varepsilon \right) \qquad \hbox{for}\;\; x\in \rn\,.
\end{equation}
Given a function $u \in L^1_{\rm loc}(\rn)$, we define
\begin{equation}\label{conv} u_{\varepsilon}= \varrho_{\varepsilon} * u,
\end{equation}
the convolution of $u$ with $ \varrho_{\varepsilon} $.

 The proof of property \eqref{11} makes use of the next  two lemmas from
\cite{gossez} and \cite{BreitCia}.

\begin{lemmaalph}[{\bf  \cite{gossez}}]\label{lemma5} \emph{Let $A$ be a finite-valued  Young function. Assume that $u\in \M (\rn)$ is such that
\begin{equation}\label{A1}
% u\in L^A(\rn) \qquad \hbox{and}\qquad
\int_{\rn} A(2\, |u(x)|)\; dx <\infty\,.
\end{equation}
Then, $u_{\varepsilon} \in C^\infty (\rn)$ for $\varepsilon>0$, and
\begin{equation}\label{A4}
\lim_{\varepsilon \to 0^+}\,  \int_{\rn} {\color{black}A\left(| u_{\varepsilon} (x)- u(x)| \right)}\; dx =0\,.
\end{equation}}
\end{lemmaalph}

\medskip
Given a function $u: \rn \to \R$ and $h>0$, we set
\begin{equation}\label{A7}
\Delta_i^h u(x) = u(x+h\, e_i) -u(x) \quad \text{for  $x\in \rn$,}
\end{equation}
and for  $i\in\{1, \dots, n\}$. Here, $e_i$ denotes the $i$-th coordinate unit vector in $\rn$.

\begin{lemmaalph}[{\bf \cite{BreitCia}}] \label{lemma6} \emph{ Let $s\in (0, 1)$ and let $A$ be a Young function. Then, there exist  constants $c=c(n)$ and $c'=c'(n)$ such that
% $u$ is such that
%\begin{equation}\label{A5}
%\int_{\rn}\int_{\rn} A\left(\frac{|u(x) - u(y)|}{|x-y|^s}\right)\, \frac{dx\,dy}{|x-y|^n} <\infty \,.
%\end{equation}
%Then, there exists $C$ such that
\begin{align}\label{A6}
&\sum_{i=1}^n \int_0^\infty \int_{\rn} A\left(c \frac{|\Delta_i^h u(x)|}{h^s}\right) \,dx\, \frac{dh}{h}
\\ \nonumber  &  \quad \quad \leq
\int_{\rn}\int_{\rn} A\left(\frac{|u(x) - u(y)|}{|x-y|^s}\right)\, \frac{dx\,dy}{|x-y|^n}
\\ \nonumber  &  \quad \quad \quad \quad \leq  \sum_{i=1}^n \int_0^\infty \int_{\rn} A\left(c' \frac{|\Delta_i^h u(x)|}{h^s}\right) \,dx\, \frac{dh}{h}
\end{align}
for every $u \in \M(\rn)$.}
%\todo[inline]{Lenka: Doesn't it follow from~\eqref{A5} that the expression of the left-hand side of~\eqref{A6} makes sense?}

\end{lemmaalph}

\begin{proof}[Proof of Theorem \ref{prop2}]
 Assume first that $s\in (0,1)$, and let   $u\in V^{s,A}_d(\rn)$. Thereby, there exists a constant $\lambda>0$ such that
 \begin{equation}\label{B1}
 \int_{\rn}\int_{\rn} A\left(\frac{|u(x)- u(y)|}{\lambda \, |x-y|^s}\right) \; \frac{dx\, dy}{|x-y|^n}<\infty\,.
 \end{equation}
 Hence, by Lemma {\color{black}\ref{lemma6}}, there exists a constant $c>0$ such that
 \begin{equation}\label{B2}
 \int_0^\infty \int_{\rn} A\left(c \frac{|\Delta_i^h u(x)|}{\lambda\, h^s}\right) \,dx \frac{dh}h {\color{black}<\infty}
 \end{equation}
for $i\in \{1, \ldots, n\}$. Fix any $i\in \{1, \ldots, n\}$.
Owing to inequality \eqref{B2},
\begin{equation}\label{B3}
  \int_{\rn} A\left(c \frac{|\Delta_i^h u(x)|}{\lambda\, h^s}\right) \,dx <\infty
 \end{equation}
% \todo[inline]{Lenka: Is it clear that this has to be true for every $h>0$? At the moment I can only see that this holds for a.e. $h$.}
for a.e. $h>0$.
 Define the function $g_i : (0,\infty) \to [0,\infty]$ as
 \begin{equation}\label{B4}
 g_i(h) = \frac 1h \, \int_{\rn} A\left(c \frac { |\Delta_i^h u(x)|}{\lambda\, h^s}\right) \; dx \quad \text{ for a.e. $h>0$.}
 \end{equation}
 Inequality  \eqref{B2} implies that
 \begin{equation}\label{B5}
 g_i\in L^1(0, \infty).
 \end{equation}
For each $h>0$,  define the function $v_i: \rn \to \R$ as
 \begin{equation}\label{B6}
 v_i(x) =\frac{c}{2 }\frac{ \Delta_i^h u(x)}{\lambda \, h^s}\qquad \hbox{for}\;\; x\in \rn\,.
 \end{equation}
Owing to inequality  \eqref{B3}, we have that
 \begin{equation}\label{B7}
 \int_{\rn}A\left(2\, |v_i(x)|\right)\; dx < \infty %\qquad \hbox{and}\qquad v\in L^A(\rn)
 \end{equation}
for a.e. $h>0$.
An application of  Lemma {\color{black}\ref{lemma5}}, with $u$ replaced by $v_i$, ensures that
 \begin{equation}\label{B8}
 \lim_{\varepsilon \to 0^+}\int_{\rn}A\left(|(v_i)_\varepsilon(x) - v_i(x)|\right)\; dx =0
 \end{equation}
 for a.e. $h>0$,
where the function $(v_i)_\varepsilon$  is defined as in \eqref{conv}.
 For $\varepsilon >0$, define the function $g_{i, \varepsilon}\colon (0,\infty) \to [0,\infty]$ as
 \begin{equation}\label{B10}
 g_{i, \varepsilon}(h) = \frac 1h \, \int_{\rn} A\left(|(v_i)_\varepsilon(x) - v_i(x)|\right) \; dx\,.
 \end{equation}
As a consequence of equation
\eqref{B8}, there exists a sequence $\{\varepsilon_j\}$ such  that
 \begin{equation}\label{B11}
 \lim_{j \to \infty} g_{i,\varepsilon_j} (h)= 0
 \end{equation}
for  a.e. $h>0$ and for every $ i\in\{1, \ldots, n\}$.
 Next,  we have that
%\todo[inline]{Lenka: Everywhere below there seems to be $\nabla$ instead of $\Delta$. I changed this without highlighting in black.}
 \begin{equation}\label{B12}
g_{i,\varepsilon_j} (h)\leq \frac 1{2h} \, \int_{\rn} A(2|(v_i)_{\varepsilon_j}(x)|) \; dx + \frac 1{2h} \, \int_{\rn} A(2|v_i(x)|) \; dx \leq  \frac 1h \, \int_{\rn} A(2|v_i(x)|) \; dx
  \end{equation}
for $h>0$,   $j \in \N$, and  $ i\in\{1, \ldots, n\}$.
Notice that the last inequality holds inasmuch as, by Jensen's inequality and the properties of $\varrho_{\varepsilon_j}$,
   \begin{equation}\label{B13}
\int_{\rn} A(2|(v_i)_{\varepsilon_j}(x)|) \; dx = \int_{\rn} A(2|\varrho_{\varepsilon_j} * v_i(x)|) \; dx
\leq  \int_{\rn} A(2|v_i(x)|) \; dx
  \end{equation}
for $h>0$.
Equations \eqref{B12} and \eqref{B13} imply that
   \begin{equation}\label{B14}
   g_{i,\varepsilon_j} (h)\leq  g_i(h)
   \end{equation}
for  $h>0$, $j \in \N$, and $i\in\{1, \ldots, n\}$.
Since, by equation \eqref{B5}, $g_i\in L^1(0, \infty)$, the dominated convergence theorem ensures that
   \begin{equation}\label{B15}
  \lim_{j \to \infty}
 \int_0^\infty g_{i,\varepsilon_j} (h)\; dh =0
 \end{equation}
for $i\in\{1, \ldots, n\}$,  whence
  \begin{align}\label{B16}
  \lim_{j \to \infty}\sum_{i=1}^n \,& \int_0^\infty \, \int_{\rn} A\left(\frac{c}{2}\frac{ \Big|  \Delta _i^h  \left(\varrho_{\varepsilon_j} *u - u\right)(x)\Big|}{\,\lambda \, h^s}\right) \; dx\, \frac{dh}h  \\ \nonumber &=  \lim_{j \to \infty}\sum_{i=1}^n \, \int_0^\infty \, \int_{\rn} A\left(\frac{c}{2} \frac{\Big| \varrho_{\varepsilon_j} * \Delta _i^h u(x) - \Delta _i^h u(x)\Big|}{\lambda \, h^s}\right) \; dx\, \frac{dh}h
=0\,.
 \end{align}
Hence, via another application of   Lemma~\ref{lemma6}, we conclude that there exists a constant $c$, depending on $\lambda$ and $n$, such that
  \begin{equation}\label{B17}
 \lim_{j \to \infty}  J_{s,A}(c(u_{\varepsilon_j} -u)) =0\,.
 \end{equation}
Thus, equation \eqref{11} is  established for $s\in (0,1)$.
\\ When $s\in(1,\infty)\setminus \N$, the conclusion follows with the same argument applied to $\nabla^{[s]} u$ and $\{s\}$ in the place of $u$ and $s$, respectively.
\end{proof}

\section{Proofs of the main results}\label{proofmain}

With the technical material of Sections \ref{boundedness} and \ref{smooth} at disposal, we are in a position to accomplish the proofs of our main results.

\begin{proof}[Proof of Theorem \ref{teo3}]
If assumption \eqref{s} is in force, namely $s\in (0,n) \setminus \N$, then \cite[Proposition 6.3 and Remark 7.3]{ACPS_frac} tell us that condition \eqref{indisp} is necessary for
any  embedding of the form \eqref{general} to hold.
\\ Thus, it is sufficient to show that, if $s\in (n, \infty) \setminus \N$, then  such an embedding fails  for every rearrangement-invariant space $Y(\rn)$. Assume, by contradiction, that there exists a rearrangement-invariant space $Y(\rn)$ which renders  embedding \eqref{general} true.
Let $\xi \in C^\infty_0(\rn)$  be  a nonnegative function such that $\xi=1$ in $\mathcal B$. For each $j\in \N$,   consider the function $u_j : \rn \to \mathbb  R$ defined as
\begin{equation}\label{100}
u_j (x)= j^{s-n} \xi \Big(\frac x j\Big) \qquad \hbox{for}\;\; x\in \rn\,.
\end{equation}
Since $u_j \in C^\infty_0(\rn)$, we have that
$$|\{ |\nabla ^k u_j|> t\}| < \infty \quad \text{ for  $t>0$,}$$
 for  $k=0, 1, \dots, [s]$.
\\
We claim that there exists a constant $c$, independent of $j$, such that
\begin{equation}\label{may30}
|\nabla ^{[s]}u_j|_{\{s\}, A, \rn} \leq c.
\end{equation}
To verify this claim, observe that
\begin{equation}\label{101}
 \nabla ^{[ s] }u_j (x) = j^{s- [ s] -n}  \nabla ^{[ s] } \xi  \Big(\frac x j\Big) = j^{\{s\} -n}   \nabla ^{[ s] } \xi  \Big(\frac x j\Big) \quad \text{for $x \in \rn$.}
\end{equation}
Therefore,
\begin{equation}\label{102}
 \frac{\Big|\nabla ^{[ s] }u_j(x) - \nabla ^{[ s] }u_j(y)\Big|}{|x-y|^{\{s\}}} =  \frac{\Big|\nabla ^{[ s] }\xi\big(\frac x j\big) - \nabla ^{[ s] }\xi\big(\frac y j\big)\Big|}{\big|\frac{x-y}j \big|^{\{s\}}}\, j^{-n} \quad \text{for $x,y \in \rn$, with $x \neq y$.}
\end{equation}
Since $\xi$ is smooth, and $j \geq 1$, the right-hand side of equation \eqref{102} is pointwise bounded by a constant $t_0$ independent of $j$.
\\
Next, since   $A$ is a Young function,  there exists a constant $c$ such that $A(t)\leq c\, t$ if $t\in [0, t_0]$.  Hence,
\begin{align}\label{103}
|\nabla ^{[ s] }u_j|_{\{s\}, A, \rn} &\leq c \, \int_{\rn} \int_{\rn} \frac{\big|\nabla ^{[ s] }u_j(x) - \nabla ^{[ s] }u_j(y)\big|}{|x-y|^{\{s\}}}\; \frac{dx \, dy}{|x-y|^n}
\\
& = c\, j^{-2n} \, \int_{\rn} \int_{\rn} \frac{\Big|\nabla ^{[ s] }\xi\big(\frac x j\big) - \nabla ^{[ s] }\xi\big(\frac y j\big)\Big|}{\big|\frac{x-y}j \big|^{\{s\}}}\; \frac{dx \, dy}{|\frac {x-y}j|^n}\nonumber
\\
& = c\,  \, \int_{\rn} \int_{\rn} \frac{\big|\nabla ^{[ s] }\xi(x) - \nabla ^{[ s]}\xi(y)\big|}{|x-y|^{\{s\}}}\; \frac{dx \, dy}{|x-y|^n}\,. \nonumber
\end{align}
The assumption that $\xi \in C^\infty_0(\rn)$  ensures that the integral on
the rightmost-hand side  of equation \eqref{103} is convergent. Hence, inequality \eqref{may30} follows.
\\ Recall that any rearrangement-invariant space $Y(\rn)$ is continuously embedded into $(L^1 + L^\infty)(\rn)$ {\color{black} --} see e.g. \cite[Theorem~6.6, Chapter~2]{BS}. Thereby, assumption \eqref{general} entails that
$$
 V^{s,A}_d(\rn)\to (L^1 + L^\infty)(\rn).
$$
Hence, from inequality \eqref{may30} we deduce that
\begin{equation}\label{may32}
\|u_j\|_{(L^1 + L^\infty)(\rn)} \leq c
\end{equation}
for some constant $c$ and
for every $j\geq 1$.
On the other hand, by
\cite[Theorem 6.2, Chapter 2]{BS},
$$\|u_j\|_{(L^1 + L^\infty)(\rn)}= \int_0^1 u_j^* (t) \;dt. $$
Thus,
\begin{align}\label{104}
\|u_j\|_{(L^1 + L^\infty)(\rn)}&= \int_0^1 u_j^* (t) \;dt = j^{s-n}\, \int_0^1 \xi^* \left(\frac t{ j^n}\right)\; dt=  j^s \, \int_0^{\frac 1{j^n}} \xi^* (\tau) \; d\tau
\\
&\geq j^s \xi^* \left(\frac 1{j^n}\right)\int_0^{\frac 1{j^n}} d\tau= j^{s-n} \, \xi^*\left(\frac 1{j^n}\right)\geq j^{s-n} \qquad \hbox{for sufficiently large $j$}\,.\nonumber
    \end{align}
Inasmuch as $s-n>0$, coupling inequality \eqref{may32} with  \eqref{104} yields a contradiction, for sufficiently large $j$.
\end{proof}

The following proposition substantiates the assertion from Section \ref{back} that $V^{s,A}_d(\rn)$ is actually a Banach space. This piece of information will be  needed for an application of the closed graph theorem in the proof of Theorem \ref{teo_cont}.

\begin{proposition}\label{banach}
Assume that  $s$ and  $A$  fulfill conditions \eqref{s} and \eqref{indisp}. Then,
the functional $\big|\nabla ^{[s]}u\big|_{\{s\},A,\rn}$ defines a norm on the space $V^{s,A}_d(\rn)$.
Moreover, the
 space $V^{s,A}_d(\rn)$, equipped with this norm, is a Banach space.
\end{proposition}

\begin{proof}
Checking that the functional $\big|\nabla ^{[s]}u\big|_{\{s\},A, \rn}$ is  a norm on  $V^{s,A}_d(\rn)$ is standard.
% In particular, either inequality \eqref{may21} or \eqref{may23} ensures that $u=0$ if  $\big|\nabla ^{[s]}u\big|_{\{s\},A, \rn}=0$.
Let us show that $V^{s,A}_d(\rn)$, equipped with this norm, is complete.
Denote by $\overline A$ a Young function  such that  $\overline A =A$ if  $A$
 fulfills condition \eqref{subcrit}, and $\overline A =B$, where $B$ is any function satisfying properties \eqref{2} and \eqref{3}, if $A$
 fulfills condition \eqref{supercrit}.
%
%Denote by $\overline A$ a Young function that agrees with $A$ globally if
% the latter fulfils condition \eqref{subcrit}. If this is not the case,
%choose  $\overline A$ in such a way that  it agrees with $A$ near zero, is dominated by $A$ near infinity, and  condition \eqref{subcrit} holds with $A$ replaced by $\overline A$.
Thus,
\begin{equation}\label{march1}
V^{s,A}_d(\rn) \to V^{s,\overline A}_d(\rn) \to \bigcap _{k=0}^{[s]}V^{k, \overline A_{\frac n{s-k}}}(\rn),
\end{equation}
where the second embedding holds thanks to embedding \eqref{may20}, applied with  $s$ subsequently replaced by $s-k$ for $k=0,1,\dots,[s]$.
Here, $ \overline A_{\frac n{s-k}}$ denotes the Young function defined as in \eqref{Ans}--\eqref{H}, with $A$ replaced by $\overline A$.
\\
Now, assume that $\{u_j\}$ is a Cauchy sequence in the space $V^{s,A}_d(\rn)$. Thanks to  embeddings \eqref{march1}, it is also a Cauchy sequence in the space $V^{s,A}(\rn)\bigcap  \Big(\bigcap _{k=0}^{[s]}V^{k, \overline A_{\frac n{s-k}}}(\rn)\Big)$,  endowed with the norm
$$\sum _{k=0}^{[s]} \|\nabla ^k u\|_{\overline A_{\frac n{s-k}}(\rn)} + \big|\nabla ^{[s]}u\big|_{\{s\},A, \rn}.$$
A customary  argument, analogous to the one showing that the classical fractional Sobolev space is a Banach space,
%
%showing that $W^{s,A}(\rn)$ is a Banach space (see e.g. {\color{black}\cite[Proposition~2.11]{Bon:19}})
tells us that $V^{s,A}(\rn)\bigcap \Big( \bigcap _{k=0}^{[s]}V^{k, \overline A_{\frac n{s-k}}}(\rn)\Big)$ is a Banach space. Hence, the sequence $\{u_j\}$ converges to some function $u \in V^{s,A}(\rn)\bigcap \Big( \bigcap _{k=0}^{[s]}V^{k, \overline A_{\frac n{s-k}}}(\rn)\Big)$. In particular,  the fact that $u \in
 \bigcap _{k=0}^{[s]}V^{k, \overline A_{\frac n{s-k}}}(\rn)$ entails that
$$
    |\{|\nabla ^k u|>t\}|<\infty\quad \text{for $t>0$,}
$$
for $k=0,1, \dots ,[s]$.
Hence, $u \in V^{s,A}_d(\rn)$, and $u_j \to u$ in $V^{s,A}_d(\rn)$.
\end{proof}

\medskip

\begin{proof}[Proof of Theorem \ref{teo_cont}]
%We begin by showing that assumption  \eqref{2.8} implies
%embedding \eqref{linf}.
%Fix any Young function $E$ and let $B$ be any Young function as in the statement of Lemma \ref{lemma1}. Define the Young function $E_0$ as
%\begin{equation}\label{4.11}
%E_0(t) =
%\begin{cases}
% 0 & \hbox{near $0$}
% \\
% E(t) & \hbox{near infinity}\,.
% \end{cases}
% \end{equation}
%By \eqref{4},
%\begin{equation*}
%  B_{\frac ns} \quad \hbox{dominates} \quad E_0 \;\; \hbox{globally}  \,.
%\end{equation*}
%Hence,
%\begin{equation}\label{4.12}
%L^{B_{\frac ns}}(\rn)\to L^{E_0}(\rn) \,.
%\end{equation}
%Since, in particular, the function $B$ satisfies assumption \eqref{3}, by the subcritical embedding \eqref{may20} applied with $A$ replaced by $B$, we have that
%\begin{equation}\label{4.13}
%V_d^{s,B} (\rn)\to L^{B_{\frac ns}}(\rn).
%\end{equation}
%On the other hand, property \eqref{2} implies that
%\begin{equation}\label{4.14}
%V_d^{s,A} (\rn)\to V_d^{s,B} (\rn) \,.
%\end{equation}
%Combining embeddings \eqref{4.12}-\eqref{4.14} yields
%\begin{equation}\label{4.15}
%V_d^{s,A} (\rn)\to L^{E_0} (\rn) \,.
%\end{equation}
% Owing to the arbitrariness of the Young function $E$, the latter embedding ensures, via Lemma~{\color{black}\ref{lemma4}}, that
%\begin{equation}\label{4.16}
%V_d^{s,A} (\rn) \subset L^\infty (\rn) \,.
%\end{equation}
We begin by showing that assertion (i) implies (ii).
Assume that condition \eqref{supercrit} holds.  Then, by Theorem \ref{bound}, inclusion \eqref{inclusion} holds.
%\begin{equation}\label{4.16}
%V_d^{s,A} (\rn) \subset L^\infty (\rn) \,.
%\end{equation}
Now, observe that the identity map from $V_d^{s,A} (\rn)$ into $L^\infty (\rn)$ has a closed graph. This claim  is   equivalent to the fact that, if $\{u_j\} \subset V_d^{s,A}(\rn)$ is a sequence such that \begin{equation}\label{may40}
    u_j\to u \qquad \hbox {in $V_d^{s,A}(\rn)$,}
\end{equation}
and
\begin{equation}\label{4.18}
    u_j \to v \qquad \hbox {in $L^\infty(\rn)$}
\end{equation}
for some functions $u\in V_d^{s,A}(\rn)$ and $v\in  L^\infty(\rn)$, then
\begin{equation}\label{4.19}
   u=v \,.
\end{equation}
To  verify that the convergences in \eqref{may40} and \eqref{4.18} imply equation \eqref{4.19}, note that, if $\overline A$ is any Young function   as in the proof of Proposition \ref{banach}, then, by equation \eqref{march1},
\begin{equation}\label{4.19'}
V_d^{s,A} (\rn)\to L^{\overline A_{\frac ns}}(\rn).
\end{equation}
Equation \eqref{may40} and embedding \eqref{4.19'}  ensure   that there exists a subsequence of $\{u_j\}$, still indexed by $j$,    such that
\begin{equation}\label{4.20}
    u_{j} \to u \qquad \hbox {a.e. in $\rn$}\,.
\end{equation}
Coupling equation \eqref{4.18} with \eqref{4.20} implies \eqref{4.19}.
\\  By Proposition \ref{banach}, $V_d^{s,A}(\rn)$ is a Banach space. Since $L^\infty(\rn)$ is also a Banach space, inclusion \eqref{inclusion} yields, via the closed graph theorem,
   the continuous embedding \eqref{linf}.
%
% Since embedding \eqref{cont} trivially implies embedding \eqref{linf}, the proof of the equivalence of assertions (i), (ii) and (iii) will be complete  once we have shown that

Let us now prove that, conversely, assertion (ii) implies (i), namely that
embedding \eqref{linf} implies   condition \eqref{supercrit}. Owing to property \eqref{rednec},   embedding \eqref{linf} entails that there exists a constant $c$ such that
\begin{equation}\label{may50}
\int_0^\infty f(\rho) \, \rho ^{-1 + \frac sn} \; d \rho = \bigg \| \int_r^\infty f(\rho) \, \rho ^{-1 + \frac sn} \; d\rho \bigg\|_{L^\infty(0, \infty)} \leq c \, \|f\|_{L^A (0, \infty)}
\end{equation}
for every nonnegative function $f\in L^A (0, \infty)$. From equation \eqref{may50} and inequality \eqref{holderconv} we deduce that
\begin{equation}\label{may51}
\|r ^{-1 + \frac sn} \|_{L^{\widetilde A}(0, \infty)} \leq c.
\end{equation}
One can verify, by the very definition of Luxemburg norm, that the finiteness of the norm in \eqref{may51} is equivalent to
$$\int_0^\infty \frac{\widetilde A(t)}{t^{1+\frac n{n-s}}}\, dt <\infty,$$
see e.g. \cite[Proof of equation (3.10)]{cianchi_ASNS}. As shown in \cite[Lemma 2.3]{cianchi-ibero}, this condition is in turn equivalent to
$$\int_0^\infty \bigg(\frac{t}{A(t)}\bigg)^{\frac s{n-s}}\, dt <\infty.$$
Hence, property \eqref{supercrit} follows.

The fact that assertion (iii) implies (ii) is trivial. Thus, it  remains to show that (ii) implies (iii). A proof of this implication relies upon inequality \eqref{201bis}, which can be established as follows.
Consider any function
$u \in V^{s,A}_d(\rn)$. If the right-hand side of inequality \eqref{201bis} is infinite, then  the inequality holds trivially.  Hence, we may assume that it is finite.  Set
\begin{equation}\label{202}
N=\left(\int_{\rn}\int_{\rn} A\left(\frac{|\nabla ^{[s]}u(x)-\nabla ^{[s]} u(y)|}{|x-y|^{\{s\}}}\right)\, \frac{dx\,dy}{|x-y|^n}\right)^{-\frac 1n},
\end{equation}
and define the function $u_N: \rn \to \mathbb R$ as
$$
u_N(x)=u\Big(\frac xN\Big)\qquad \hbox{for $x\in \rn$}\,.
$$
Hence,
$$
\nabla ^{[s]} u_N(x)= N^{-[s]}\nabla ^{[s]} u\Big(\frac xN\Big)\qquad \hbox{for $x\in \rn$}\,.
$$
Clearly,
\begin{equation}\label{204}
\|u_N\|_{L^\infty(\rn)} = \|u\|_{L^\infty(\rn)}.
\end{equation}
Moreover,
for every  $\lambda >0$, one has, by a change of variables, % $x' = \frac xN, y'=\frac yN$,
\begin{align}\label{205}
    \int_{\rn}\int_{\rn} &
    A\left(\frac{|\nabla ^{[s]} u_N(x) - \nabla ^{[s]} u_N(y)|}{\lambda \, |x-y|^{\{s\}}}\right) \, \frac{dx\,dy}{|x-y|^n}
        \\
    &=
    \int_{\rn}\int_{\rn}
    A\left(\frac{|  \nabla ^{[s]}u\left(\frac xN\right) -  \nabla ^{[s]}u\left(\frac yN\right)|}{\lambda \, |\frac xN-\frac yN|^{\{s\}}}N^{-[s]-\{s\}}\right) \, \frac{dx\,dy}{|\frac xN - \frac yN|^n}N^{-n}\nonumber
        \\
    & =\int_{\rn}\int_{\rn} A\left(\frac{|  \nabla ^{[s]}u(x') -  \nabla ^{[s]}u(y')|}{\lambda \, |x'-y'|^{\{s\}}}N^{-s}\right) \, \frac{dx'\,dy'}{|x'-y'|^n}N^n\,. \nonumber
\end{align}
The  choice $\lambda = N^{-s}$ in \eqref{205} yields
\begin{align}\label{206}
\int_{\rn}\int_{\rn}
 A\left(\frac{|\nabla ^{[s]} u_N(x) - \nabla ^{[s]} u_N(y)|}{N^{-s} \, |x-y|^{\{s\}}}\right) \, \frac{dx\,dy}{|x-y|^n} = 1.
%\frac{\int_{\rn}\int_{\rn}
% A\left(\frac{| u_N(x) - u_N(y)|}{|x-y|^s}\right) \, \frac{dx\,dy}{|x-y|^n}  }{\int_{\rn}\int_{\rn}
% A\left(\frac{| u_N(x) - u_N(y)|}{ |x-y|^s}\right) \, \frac{dx\,dy}{|x-y|^n}} =1\,.
 \end{align}
 Hence, the definition of the seminorm $| \, \cdot \, |_{s, A, \rn}$ implies that
 \begin{equation}\label{207}
 | u_N|_{s, A, \rn}\leq N^{-s}\,.
 \end{equation}
From embedding \eqref{linf}   and equations \eqref{204} and \eqref{207} one can infer that
\begin{equation}\label{208}
 \| u\|_{L^\infty(\rn)}\leq c \, N^{-s}
 \end{equation}
for some constant $c$. Hence,
 inequality \eqref{201bis} follows.

We are now ready to complete the proof by showing
that assertion (iii) follows from (ii).  This goal will be achieved on proving that any function $u \in  V_d^{s,A}(\rn)$ agrees  with a continuous function a.e.  in $\rn$. To this purpose, we make use of Theorem \ref{prop2}, and pick a number $\lambda >0$ and a sequence  $\{u_j\}\subset C^\infty(\rn)$ fulfilling property \eqref{11}.  Fix $\varepsilon >0$. Then, there exists $j_\varepsilon \in \N$ such that
\begin{equation}\label{T2}
J_{\{s\},A}\left( \frac{\nabla ^{[s]}u_j -\nabla ^{[s]} u}{\lambda}\right)<\varepsilon \qquad \hbox{for}\;\; j\geq j_\varepsilon \,.
\end{equation}
By inequality \eqref{201bis} applied with $u$ replaced by  $(u_i -u)/\lambda$ and $(u_j -u)/\lambda$,  with $i,j \geq j_\varepsilon$,
\begin{align}\label{T3}
\|u_i -u_j\|_{C^0{(\rn)}}&= \|u_i -u_j\|_{L^\infty{(\rn)}} \leq \|u_j -u\|_{L^\infty(\rn)} + \|u_i -u\|_{L^\infty (\rn)}
\\
&
\leq c\, \lambda \left(J_{\{s\},A} \left(\frac{\nabla ^{[s]} u_i -\nabla ^{[s]} u}{\lambda} \right)^{\frac sn} + J_{\{s\},A} \left(\frac{\nabla ^{[s]}u_j -\nabla ^{[s]}u}{\lambda} \right)^{\frac sn}\right)<2c\lambda\varepsilon^{\frac sn}.\nonumber
\end{align}
 Equation \eqref{T3} implies, in particular, that
$u_j \to u$  in $L^\infty(\rn)$.
Moreover, equation \eqref{T3} tells us that
$\{u_j\}$ is a Cauchy sequence in the Banach space $C^0(\rn)$. Thus, there exists
a function $\overline{u}\in C^0(\rn)$ such that $u_j\to \overline u$ in  $C^0(\rn)$, and, hence, in $L^\infty(\rn)$. By the uniqueness of the limit, $\bar{u}=u$ a.e. on $\rn$.
%
%
%
%
%Finally, inequality \eqref{201bis} is a consequence of Lemma \ref{prop1}.
\end{proof}

\begin{lemma}\label{lemma3}
Assume that  $s$ and  $A$  fulfill conditions \eqref{s} and \eqref{indisp}. Let $B$ be another Young function such that
%
%Let $s\in (0,n)\setminus \N$ and let $A$ and $B$ be Young functions such that
%\begin{equation}\label{1.1}
%\int_0\left( \frac t{A(t)} \right)^{\frac s{n-s}} \; dt < \infty\,,%\qquad \int_0\left( \frac t{B(t)} \right)^{\frac s{n-s}} \; dt < \infty\,,
%\end{equation}
%and
\begin{equation}\label{1.2}
A(t) \simeq B(t) \qquad \hbox{near $0$}\,.
\end{equation}
Then,
\begin{equation}\label{1.3}
\left (L^\infty \cap L(A, \tfrac ns) \right)(\rn) =  \left ( L^\infty \cap  L(B, \tfrac ns) \right)(\rn) \,,
\end{equation}
up to equivalent norms.
\end{lemma}

\begin{proof} Owing to assumptions \eqref{indisp}  and \eqref{1.2}, we also have that
$$ \int_0\left( \frac t{B(t)} \right)^{\frac s{n-s}} \; dt < \infty.$$
Define the Young functions $E_A$ and $E_B$ in such a way that
$$
E_A(t) \simeq
\begin{cases}
 A(t) & \hbox{near $0$}
 \\
 \infty & \hbox{near infinity}
 \end{cases}
 \qquad\qquad\qquad
 E_B(t)\simeq
 \begin{cases}
 B(t) & \hbox{near $0$}
 \\
 \infty & \hbox{near infinity,}
 \end{cases}
$$
and the function
$ \phi : (0, \infty) \to [0, \infty)$ as in \eqref{1.5}.
%\begin{equation}\label{1.5}
%\xi (r)= \min \Big\{1, r^{-\frac sn}\Big\}\qquad \hbox{for $r>0$}\,.
%\end{equation}
By the equivalence of the norms \eqref{normsum} and \eqref{normequiv},%\cite[Proposition  2.1]{cianchi-ibero},
\begin{equation}\label{july100}
    \|u\|_{\big (L^\infty \cap  L(A, \frac ns)\big)(\rn)}\approx \|u^* (r)\, \phi (r)\|_{L^{E_A} (0, \infty)}\,,
\end{equation}
and
\begin{equation}\label{july101}
    \|u\|_{\big (L^\infty  \cap  L(B, \frac ns)\big)(\rn)}\approx \|u^* (r)\, \phi (r)\|_{L^{E_B} (0, \infty)},
\end{equation}
with equivalence constants independent of $u\in \M(0,\infty)$.
Owing to assumption \eqref{1.2}, the Young functions $E_A$ and $E_B$ are globally equivalent. Thus, the norms on the right-hand sides of equations \eqref{july100} and \eqref{july101} are equivalent. Hence,
\begin{equation}\label{1.7}
    \|u\|_{\big ( L^\infty \cap L(A, \frac ns)  \big)(\rn)}\approx
 \|u\|_{\big ( L^\infty \cap L(B, \frac ns) \big)(\rn)}\,,
\end{equation}
as well,
with equivalence constants independent of $u\in \M(0,\infty)$. Equation \eqref{1.3} follows.
 \end{proof}

\begin{proof}[Proof of Theorem \ref{teo1}]
Let $B$ be any Young function  satisfying  properties \eqref{2} and \eqref{3}.
 By embedding \eqref{may22},
\begin{equation}\label{4.21}
V_d^{s,A} (\rn)\to V_d^{s,B} (\rn) \to L(\widehat{B}, \tfrac ns )(\rn) \,,
\end{equation}
where $\widehat{B}$  is the Young function defined as in \eqref{E:1} --\eqref{E:2}, with $A, a, \widehat{a}$ replaced by $B, b, \widehat{b}$, respectively.
\\
From embeddings \eqref{linf} and \eqref{4.21} one infers that
\begin{equation}\label{4.22}
V_d^{s,A} (\rn)\to \big(L^\infty  \cap L(\widehat{B}, \tfrac ns )\big)(\rn) \,.
\end{equation}
Since $B \simeq A$ near zero, one can verify that
\begin{equation}\label{4.23}
\widehat{B} \simeq  \widehat{A} \quad \hbox{ near zero}\,.
\end{equation}
Hence, Lemma \ref{lemma3}, applied with $A$ and $B$ replaced with $\widehat{A}$ and $\widehat{B}$, tells us that
\begin{equation}\label{4.24}
\big(L^\infty  \cap L(\widehat{B}, \tfrac ns )\big)(\rn) = \big(L^\infty \cap L(\widehat{A}, \tfrac ns )\big)(\rn),
\end{equation}
up to equivalent norms.
Embedding \eqref{2.9} follows from equations \eqref{4.22} and \eqref{4.24}.
\\
As far as the optimality of the space
$\big (L^\infty  \cap L(\widehat{A}, \tfrac ns )\big)(\rn)$ is  concerned, assume that embedding \eqref{general} holds for some
  rearrangement-invariant space $Y(\rn)$.
% such that
%\begin{equation}\label{4.25}
%V_d^{s,A}(\rn) \to Y(\rn)\,.
%\end{equation}
We have to show that
\begin{equation}\label{4.26}
\big(L^\infty  \cap L(\widehat{A}, \tfrac ns )\big)(\rn) \to Y(\rn)\,.
\end{equation}
By property \eqref{rednec}, embedding \eqref{general} implies that there exists a constant $c$ such that
\begin{equation}\label{4.28}
\bigg \| \int_r^\infty f(\rho) \, \rho ^{-1 + \frac sn} \; d\rho \bigg\|_{Y(0, \infty)} \leq c \, \|f\|_{L^A (0, \infty)}
\end{equation}
for every function $f\in L^A(0, \infty)$. Inequality \eqref{4.28} entails that
\begin{equation}\label{4.27}
\big( L^\infty  \cap L(\widehat{A}, \tfrac ns )\big)(\rn)\subset Y(\rn)\,.
\end{equation}
This can be verified
via the same argument as in the proof of optimality in \cite[Theorem 1.1]{cianchi-ibero}.
Thanks to
 property \eqref{inclusion-embedding}, inclusion \eqref{4.27} is equivalent to  embedding \eqref{4.26}.
\end{proof}

\begin{proof}[Proof of Theorem \ref{teo2}] Let $B$ be a Young function as in the  statement of Lemma \ref{lemma1}. Embedding \eqref{4.30}, with $A$ replaced by $B$, and embedding
 \eqref{4.22} implies  that
\begin{equation}\label{4.31}
V_d^{s,A} (\rn)\to\big( L^\infty  \cap L^{B_{\frac ns}}\big)(\rn) \,.
\end{equation}
Since $B \simeq A$ near zero,
\begin{equation}\label{4.32}
B_{\frac ns} \simeq A_{\frac ns} \quad \hbox{ near zero}\,.
\end{equation}
Hence, via an application of \cite[Lemma 5.1]{ACPS_entire} we deduce  that
\begin{equation}\label{4.33}
\big( L^\infty  \cap L^{B_{\frac ns}}\big)(\rn) \to L^{A_{\frac ns}}(\rn) \,.
\end{equation}
Embedding \eqref{2.10} follows from \eqref{4.31} and \eqref{4.33}.
\\  It remains to prove that   $L^{A_{\frac ns}}(\rn)$ is the optimal   Orlicz target space in \eqref{2.10}.
Assume that $E$ is a Young function such that
\begin{equation}\label{4.34}
V_d^{s,A} (\rn)\to L^E (\rn)\,.
\end{equation}
We have to show that
\begin{equation}\label{june100}
L^{A_{\frac ns}}(\rn)\to L^E (\rn).
\end{equation}
Thanks to property \eqref{rednec}, there exists a constant $c$ such that
\begin{equation}\label{4.35}
\bigg \| \int_r^\infty f(\rho) \, \rho ^{-1 + \frac sn} \; d\rho \bigg\|_{L^E(0, \infty)} \leq c \, \|f\|_{L^A (0, \infty)}
\end{equation}
for every function $f\in L^A(0, \infty)$.
By \cite[Lemma 1]{cianchi-ibero}, $L^{A_{\frac ns}}(0, \infty)$ is the optimal Orlicz target space in \eqref{4.35}. Thus,
$
L^{A_{\frac ns}}(0, \infty)\to L^E (0, \infty)$, and this embedding is equivalent to embedding \eqref{june100}.
\end{proof}

%
%\begin{lemma}[\bf Reduction Principle]\label{T:reduction_principle}%{\rm{\bf [Higher-order reduction principle]}}
%Let $n\in \N$ and  $s \in (0,n)\setminus
%\N$.  Let $A$ be a Young function and let $Y(\rn)$ be a rearrangement-invariant space. The following statements are equivalent:
%
%\begin{itemize}
%
%\item[i)] {
%There exists a constant $C$ such that
%\begin{equation}\label{E:i}
%\|u\|_{Y(\rn)} \leq C  \big|\nabla ^{[s]}u\big|_{\{s\},A, \R^n}
%\end{equation}
%for every function $u \in V^{s,A}_d(\rn)$.
%}
%
%
%\item[ii)] {
%There exists a constant $C'$ such that
%\begin{equation}\label{E:ii}
%	 \bigg\|\int_{r}^{\infty}f(\varrho)\varrho^{-1+\frac{s}{n}}\,d\varrho \bigg\|_{\overline{Y}(0, \infty)}
%		\leq C'\|f\|_{L^A(0, \infty)}
%\end{equation}
% for every  function $f\in L^A(0,\infty)$.
% }
% \end{itemize}
%\end{lemma}

\medskip
\begin{proof}[Proof of Theorem \ref{T:reduction_principle}] The fact that inequality \eqref{E:i} implies inequality \eqref{E:ii} is stated in property \eqref{rednec}, and established  in  \cite[Lemmas 6.5 and 7.6]{ACPS_frac}.
\\
Let us prove the reverse implication. Assume that inequality \eqref{E:ii} holds for some rearrangement-invariant space $Y(\rn)$
We distinguish two cases, corresponding to the subcritical regime \eqref{subcrit} and the supercritical regime \eqref{supercrit}.
\\
If  condition \eqref{subcrit} is in force, then  inequality  \eqref{may23} holds. Hence, by property  \eqref{rednec},
 \begin{equation}\label{4.38}
\bigg \| \int_r^\infty f(\rho) \, \rho ^{-1 + \frac sn} \; d\rho \bigg\|_{L(\widehat{A}, \frac ns)(0, \infty)} \leq C \|f\|_{L^A(0, \infty)}
\end{equation}
for every $f \in L^A(0, \infty)$. Moreover,
the target space $L(\widehat{A}, \frac ns)(0, \infty)$ is optimal in inequality \eqref{4.38} among all rearrangement-invariant spaces {\color{black}--} see the proof of the optimality in \cite[Theorem 1.1, Part I]{cianchi-ibero}.
% Assume that $\overline{Y}(0, \infty)$ is such that
%\begin{equation}\label{4.40}
%\bigg \| \int_r^\infty f(\varrho) \, \varrho ^{-1 + \frac sn} \; d\varrho \bigg\|_{\overline{Y}(0, \infty)} \leq C \, \|f\|_{L^A (0, \infty)}
%\end{equation}
%for every function $f\in L^A(0, \infty)$.
%Inasmuch as inequality \eqref{E:ii} is nothing but \eqref{4.38}, with the space $L(\widehat{A}, \frac ns)(0, \infty)$ replaced by $\overline Y(0, \infty)$,
This optimality
 ensures that $L(\widehat{A}, \frac ns)(0, \infty) \to \overline Y(0, \infty)$. Hence,
\begin{equation}\label{4.41}
 \|u\|_{Y(\rn)} \leq c \|u\|_{L(\widehat{A}, \frac ns)(\rn)}
\end{equation}
for some constant $c$ and
for every $u \in L(\widehat{A}, \frac ns)(\rn)$.
Inequality \eqref{E:i} follows  from \eqref{may23}  and \eqref{4.41}.
%
%there exists a constant $c$  such that
% \begin{equation}\label{4.42}
%\|u\|_{Y(\rn)}\leq c \big|\nabla ^{[s]}u\big|_{\{s\},A, \R^n}
%\end{equation}
%for every $u \in V_d^{s,A}(\rn)$. This proves inequality \eqref{E:i}.
%Moreover,
%\begin{equation}
% V_d^{s,A}(\rn) \to Y(\rn)\,.
%\end{equation}
\\
Assume next that
 condition \eqref{supercrit} holds. Then, Theorem \ref{teo1}  provides us with inequality \eqref{2.9norm}, which, coupled with
%By \cite[Theorem 6.2]{ACPS_frac},
% \begin{equation}\label{4.43}
% \|u\|_{\left(L(\widehat{A}, \frac ns)\cap L^\infty\right)(\rn)} \leq C\, \big|\nabla ^{[s]}u\big|_{\{s\},A, \R^n}
%\end{equation}
% for every function $u\in V_d^{s,A}(\rn)$.
 property \eqref{rednec},  yields the inequality
 \begin{equation}\label{4.44}
\bigg \| \int_r^\infty f(\rho) \, \rho ^{-1 + \frac sn} \; d\rho \bigg\|_{\left(L^\infty \cap L(\widehat{A}, \frac ns)\right)(0, \infty)} \leq  c\, \|f\|_{L^A(0, \infty)}
\end{equation}
for some constant $c$ and for every $f \in L^A(0, \infty)$.
 The space $\big(L^\infty \cap L(\widehat{A}, \frac ns)\big)(0, \infty)$ is optimal in inequality  \eqref{4.44} among all rearrangement-invariant spaces {\color{black}--} see the  proof of the optimality in \cite[Theorem 1.1, Part II]{cianchi-ibero}.
This optimality
 guarantees that $\big(L^\infty \cap L(\widehat{A}, \frac ns)\big)(0, \infty) \to \overline Y(0, \infty)$, whence
\begin{equation}\label{4.41bis}
 \|u\|_{Y(\rn)} \leq  c \|u\|_{\left( L^\infty \cap L(\widehat{A}, \frac ns)\right)(\rn)}
\end{equation}
for some constant $c$ and for every $u \in  \big(L^\infty \cap L(\widehat{A}, \frac ns)\big)(\rn)$. Inequality \eqref{E:i}  is a consequence of
 \eqref{2.9norm} and \eqref{4.41bis}.
%   we infer that there exists a constant $c$   such that
%$$
%\|u\|_{Y(\rn)}\leq c\big|\nabla ^{[s]}u\big|_{\{s\},A, \R^n}
%$$
%for every $u \in V_d^{s,A}(\rn)$. Hence, inequality \eqref{E:i}  follows.
\end{proof}

\section*{Compliance with Ethical Standards}\label{conflicts}

\smallskip
\par\noindent
{\bf Funding}. This research was partly funded by:
\\ (i) GNAMPA   of the Italian INdAM - National Institute of High Mathematics (grant number not available)  (A. Alberico, A. Cianchi);
\\ (ii) Research Project   of the Italian Ministry of Education, University and
Research (MIUR) Prin 2017 ``Direct and inverse problems for partial differential equations: theoretical aspects and applications'',
grant number 201758MTR2 (A. Cianchi);
\\  (iii)  Grant P201/21-01976S of the Czech Science Foundation  (L. Pick);
\\ (iv) Primus research programme PRIMUS/21/SCI/002 of Charles University (L. Slav\'ikov\'a).

\bigskip
\par\noindent
{\bf Conflict of Interest}. The authors declare that they have no conflict of interest.

%\todo[inline]{
%The Riesz-Fischer Property:
%
%$V$ is a Banach space if and only if it follows from $\sum_{n=1} ^\infty \|f_n\|<\infty$ (all $f_n$  in $V$) that the partial sums $s_n=\sum_1^n f_k$ have a norm limit in $V$ (as $n \to \infty$).}

%%%%%%%%%%%%%%%%%%%%%%% Bibliography %%%%%%%%%%%%%%%%%%%%%%%%

\bigskip

%\todo[inline]{Angela: I have added some references chosen by Andrea from the .doc file that I sent  you in the last mail.
%They are coloured magenta.}

%\todo[inline]{Lenka: I have added the paper by Gossez that we are using.}

\end{document}